\documentclass[10pt]{article}
\advance\oddsidemargin by -1.0cm
\advance\evensidemargin by -1.0cm
\advance\topmargin by -1.0cm
\usepackage{latexsym,amssymb,yfonts} %
\usepackage{amsmath}
\usepackage{amsthm}
\usepackage{mathtools}
\usepackage{tikz}
\usepackage{enumerate}
\usepackage{psfrag}
\usepackage{graphicx}
\usepackage[colorlinks=true,linkcolor=black,urlcolor=black,citecolor=black,bookmarks]{hyperref}
\usepackage{multirow}
\usepackage{booktabs}
\numberwithin{equation}{section}
\newtheorem{theorem}{Theorem}[section]
\newtheorem{prop}{Proposition}[section]

\newtheorem{cor}{Corollary}[section]

\newtheorem{lemma}{Lemma}[section]

\theoremstyle{definition}
\newtheorem{df}{Definition}[section]

\newtheorem{definition}{Definition}[section]

\theoremstyle{remark}
\newtheorem{rem}{Remark}[section]

\newtheorem*{remark*}{Remark}

\numberwithin{table}{section}
\newcommand{\bdm}{\begin{displaymath}}
\newcommand{\edm}{\end{displaymath}}
\newcommand{\be}{\begin{equation}}
\newcommand{\ee}{\end{equation}}

\newcommand{\cyclic}[1]{\stackrel{\scriptsize #1}{\mathfrak{S}}}
\newcommand{\R}{\mathbb{R}}

\renewcommand{\H}{\mathcal{H}}
\newcommand{\V}{\mathcal{V}}

\newcommand\extalg{%
  \newlength{\len}%
  \settoheight{\len}{V}%
  \mathbin{%
    \resizebox{0.93\len}{0.93\len}{$\wedge$}%
    \kern-0.1em%
  }}%
\newcommand{\intprod}{\mathbin{\hbox to 0.7ex{%
      \kern-0.3ex
      \vrule height0.0777ex width0.971ex depth0ex
      \kern-0.055ex
      \vrule height1.165ex width0.0777ex depth0ex\hss}}%
}%
\newcommand{\Rc}{\mathcal R}
\newcommand{\Rcpar}{\mathcal{R}_{\mathrm{par}}}
\newcommand{\Rcperp}{\mathcal{R}_{\perp}}
\newcommand{\G}{\mathcal G}

\newcommand{\Sc}{\mathcal S}

\newcommand{\owedge}{%
  \mathbin{%
    \mathchoice%
      {\buildcircleland{\displaystyle}}%
      {\buildcircleland{\textstyle}}%
      {\buildcircleland{\scriptstyle}}%
      {\buildcircleland{\scriptscriptstyle}}%
  }%
}
\newcommand\buildcircleland[1]{%
  \begin{tikzpicture}[baseline=(X.base), inner sep=0, outer sep=0]
    \node[draw,circle] (X)  {$#1\land$};
  \end{tikzpicture}%
}

\begin{document}

\title{Curvature Properties of $3$-$(\alpha,\delta)$-Sasaki Manifolds}

\author{Ilka Agricola, Giulia Dileo, and Leander Stecker}
\date{}

\maketitle

\begin{abstract}
We investigate curvature properties of $3$-$(\alpha,\delta)$-Sasaki manifolds, 
a special class of almost $3$-contact metric manifolds generalizing $3$-Sasaki manifolds (corresponding to $\alpha=\delta=1$) that admit a canonical metric connection with skew torsion and define a Riemannian submersion over a quaternionic K\"ahler manifold with vanishing, positive or negative scalar curvature, according to $\delta=0$, $\alpha\delta>0$ or $\alpha\delta<0$. We shall investigate both the
 Riemannian curvature and the curvature of the canonical connection, with particular focus on their 
 curvature operators, regarded as symmetric endomorphisms of the space of $2$-forms. We describe their spectrum, find distinguished eigenforms, and study the conditions of strongly definite curvature in the sense of Thorpe.
\end{abstract}

%\vspace{1cm}

\tableofcontents

%\vspace{1cm}

\pagestyle{headings}

%\bigskip\noindent
%\texttt{https://arxiv.org/abs/1804.06700}, to appear in Adv.\,Geom.

\phantom{x}

\vspace{1cm}

\medskip
\noindent
{\small

}

%\phantom{x}
%\newpage

%-----------------------------------------------------------------------------------------

%!TeX root=ADS3adCurv-bb.tex

\section{Preliminaries}
%------------------------
\subsection{Introduction}
%--------------------------
%
The present paper is devoted to the curvature properties of $3$-$(\alpha,\delta)$-Sasaki manifolds,  both of the Riemannian connection and the canonical connection and, most importantly, their interaction. We will be particularly concerned with the curvature operators, regarded as symmetric endomorphisms of the space of $2$-forms, in order to investigate their spectrum, find distinguished eigenforms, and study the conditions of strongly definite curvature in the sense of Thorpe.

\emph{$3$-$(\alpha,\delta)$-Sasaki manifolds} are a special class of almost $3$-contact metric manifolds. They were introduced in \cite{AgrDil} as a generalization of $3$-Sasaki manifolds (corresponding to $\alpha=\delta=1$), and as a subclass of \emph{canonical} almost $3$-contact metric manifolds, characterized by admitting a canonical metric connection with totally skew-symmetric torsion (skew torsion for brief). The vanishing of the coefficient $\beta:=2(\delta-2\alpha)$ defines \emph{parallel} $3$-$(\alpha,\delta)$-Sasaki manifolds, for which the canonical connection parallelizes all the structure tensor fields.
%They are distinguished into three classes: degenerate, positive and negative, according to $\delta=0$, $\alpha\delta>0$ or $\alpha\delta<0$, respectively.
The geometry of $3$-$(\alpha,\delta)$-Sasaki manifolds has been further investigated in \cite{ADS20}, where it was shown that they admit a locally defined Riemannian submersion over a quaternionic K\"ahler manifold with vanishing, positive or negative scalar curvature, according to $\delta=0$, $\alpha\delta>0$ or $\alpha\delta<0$. These coincide, respectively, with the defining conditions of \emph{degenerate}, \emph{positive} and \emph{negative} $3$-$(\alpha,\delta)$-Sasaki structures, which are all preserved by a special type of deformations, namely $\H$-homothetic deformations.
 The vertical distribution of the \emph{canonical submersion}, which turns out to have totally geodesic leaves, coincides with the $3$-dimensional distribution spanned by the three Reeb vector fields $\xi_i$, $i=1,2,3$, of the structure. The canonical connection plays a central role in this picture, as it preserves both the vertical and the horizontal distribution, and in fact, when applied to basic vector fields, it projects onto the Levi-Civita connection of the quaternionic K\"ahler base space. Beyond this introduction, the remaining part of Section 1 will be devoted to a short review of the notions and results needed in this work.

In Section \ref{section2}, we will see how the canonical curvature operator $\Rc$ is related to the Riemannian curvature operator $\Rc^{g_N}$ of the qK base space of the canonical submersion $\pi\colon M\to N$. Introducing a suitable decomposition of $\Rc$, we show that if $\Rc^{g_N}$ is non-negative, resp. non-positive, then so is the operator $\Rc$, provided that $\alpha\beta\geq0$ for non-negative definiteness (\autoref{semidefinite}). The decomposition of the operator $\Rc$ also allows to determine a set of six orthogonal eigenforms of $\Rc$, distinguished into two triples:  $\Phi_i-\xi_{jk}$, and $\Phi_i+(n+1)\xi_{jk}$, where $(ijk)$ denotes an even permutation of $(123)$, $\Phi_i$ are the fundamental $2$-forms of the structure, and $\xi_{jk}:=\xi_j\wedge\xi_k$. %The corresponding eigenvalues are given by $2\alpha\beta(n+2)$, $\dim M=4n+3$, and $0$.

The goal of \autoref{eigenvalues} is to interpret both triples $\Phi_i-\xi_{jk}$ and $\Phi_i+(n+1)\xi_{jk}$ as eigenforms, not only of $\Rc$, but also of the Riemannian curvature operator $\Rc^g$ of $M$. We show that them being eigenforms of $\Rc^g$ provides necessary and sufficient conditions for $M$ to be Einstein, which precisely happens when $\delta=\alpha$ or $\delta=(2n+3)\alpha$, with $\dim M=4n+3$ (\autoref{theo_einstein_case}). The result is obtained by taking into account the relation between the operators $\Rc$ and $\Rc^g$, involving two further symmetric operators $\G_T$ and $\Sc_T$ defined by means of the torsion of the canonical connection.

\autoref{section definiteness} is devoted to the investigation of conditions of strong definiteness for the  Riemannian curvature of a $3$-$(\alpha,\delta)$-Sasaki manifold. Recall that a Riemannian manifold $(M,g)$ is said to have
strongly positive curvature if for some $4$-form $\omega$ the modified symmetric operator $\Rc^g+\omega$ is positive definite. On the one hand, this weakens the condition of positive definiteness of the Riemannian curvature operator ($\mathcal{R}^g> 0$), which forces the Riemannian manifold to be diffeomorphic to a space form \cite{Bohm-W}. On the other hand, this provides a stronger condition than positive sectional curvature as, for any $2$-plane $\sigma$, $\operatorname{sec}(\sigma)=\langle(\mathcal{R}^g+\omega)(\sigma),\sigma\rangle$. The method of modifying the curvature operator by a $4$-form was originally introduced by Thorpe \cite{Thorpe71}, and then developed by various authors (\cite{Zoltek}, \cite{Put}, \cite{StrPosCurv}). In the same way one can introduce a notion of strongly non-negative curvature.  Considering a $3$-$(\alpha,\delta)$-Sasaki manifold $M$ with canonical submersion $\pi\colon M\to N$, we determine sufficient conditions for strongly non-negative and strongly positive curvature (\autoref{stronglypos}). We require a sufficiently large quotient $\delta/\alpha\gg 0$, together with strongly non-negative or strongly positive curvature for the quaternionic K\"ahler base space $N$. Suitable $4$-forms modifying  the Riemannian curvature operator $\Rc^g$ of $M$ are constructed using the pullback $\pi^*\omega$ of a $4$-form $\omega$ which modifies the operator $\Rc^{g_N}$, and the $4$-form $\sigma_T=\frac12 dT$, $T$ being the torsion of the canonical connection; this $4$-form is known to be a measure of the non-degeneracy of the torsion $T$, which explains its appearance in this context. We discuss the case of homogeneous $3$-$(\alpha,\delta)$-Sasaki manifolds fibering over symmetric quaternionic K\"ahler spaces of compact type (Wolf spaces) and their non-compact duals. A construction of these spaces was given in \cite{ADS20}, providing a classification in the compact case ($\alpha\delta>0$). In this case, we show that if $\alpha\beta\geq0$, then the manifold is strongly non-negative. Strong positivity is much more restrictive, as the only spaces admitting a homogeneous structure with strict positive sectional curvature are the $7$-dimensional Aloff-Wallach-space $W^{1,1}$, the spheres $S^{4n+3}$, and real projective spaces $\mathbb{R}P^{4n+3}$. For these spaces, assuming $\alpha\beta>0$,  we provide explicit $4$-forms modifying the Riemannian curvature operator to obtain strongly positive curvature (\autoref{StrictPos}). In \autoref{inhomogeneous} we show strong positive curvature for a class of inhomogeneous $3$-$(\alpha,\delta)$-Sasaski manifold obtained by $3$-Sasaki reduction, compare \cite{BGM94} and \cite{Dearricott04}.

\bigskip
\noindent\textbf{Acknowledgements.} We strongly appreciate the referee's thorough read-through of our manuscript, greatly improving its quality. The second author was partially supported by the National Group for Algebraic and Geometric Structures, and their 
Applications (GNSAGA – INdAM).

%---------------------------------------------------------------------------------------------
\subsection{Curvature endomorphisms and strongly positive curvature}
%---------------------------------------------------------------------------------------------
We review notations and established properties of connections with skew torsion and their curvature. We refer to \cite{Ag} for further details.

Let $(M,g)$ be a Riemannian manifold, $\dim M=n$. A metric
connection $\nabla$ is said to have skew torsion if the $(0,3)$-tensor field $T$ defined by
\[T(X,Y,Z)=g(T(X,Y),Z)=g(\nabla_XY-\nabla_YX-[X,Y],Z)\]
is a $3$-form. Then $\nabla$ and the Levi-Civita connection
$\nabla^g$ are related by
$
	\nabla_XY=\nabla^g_XY+\frac{1}{2}T(X,Y),
$
and $\nabla$ has the same geodesics as $\nabla^g$. Assume further that $T$ is parallel, i.e. $\nabla T=0$.
Typical examples of manifolds admitting metric connections with parallel skew torsion include Sasaki, $G_2$, nearly K\"ahler and several others (see also the recent paper \cite{CleyMorSemm}).

The fact that $\nabla T=0$ implies  $dT=2\sigma_T$, where $\sigma_T$ is the $4$-form defined by
\[
\sigma_T(X,Y,Z,V)=g(T(X,Y),T(Z,V))+g(T(Y,Z),T(X,V))+g(T(Z,X),T(Y,V)).
\]
Furthermore, the curvature tensor $R(X,Y,Z,V)=g([\nabla_X,\nabla_Y]Z-\nabla_{[X,Y]}Z,V)$ of a connection with parallel skew torsion $\nabla$ satisfies the first Bianchi identity
\begin{equation}\label{bianchi}
\cyclic{XYZ}R(X,Y,Z,V)=\sigma_T(X,Y,Z,V)=\frac12 dT(X,Y,Z,V)
\end{equation}
which implies the pair symmetry
\begin{equation}\label{pairs}
R(X,Y,Z,V)=R(Z,V,X,Y).
\end{equation}
These identities trivially apply to the Levi-Civita connection $\nabla^g$ of $(M,g)$ and its curvature $R^g$. The Riemannian curvature $R^g$ is related to $R$ by
\begin{equation}\label{Rg-R}
R^g(X,Y,Z,V)=R(X,Y,Z,V)-\frac14 g(T(X,Y),T(Z,V))-\frac14\sigma_T(X,Y,Z,V).
\end{equation}
%
%Here, the Riemannian curvature is defined by $R^g(X,Y)=[\nabla^g_X,\nabla^g_Y]-\nabla^g_{[X,Y]}$, and the associated $(0,4)$-tensor field by $R^g(X,Y,Z,V)\coloneqq g(R^g(X,Y)Z,V)$, so that the sectional curvature of a
%Here, the Riemannian curvature is defined as the $(0,4)$-tensor field by $R^g(X,Y,Z,V)\coloneqq g([\nabla^g_X,\nabla^g_Y]Z-\nabla^g_{[X,Y]}Z,V)$, so that the sectional curvature of a $2$-plane spanned by orthonormal vectors $X,Y$ is
%\[
%\operatorname{sec}(X,Y)=R^g(X,Y,Y,X)%=g(R^g(X,Y)Y,X).
%\]
%We adopt the same conventions for the curvature $R$ of the canonical connection.
%

Recall that, given a Riemannian manifold $(M,g)$, at each point $x\in M$ the space $\Lambda^p T_xM$ of $p$-vectors
of $T_xM$ can be endowed with the inner product defined by
\[
\langle u_1\wedge\ldots \wedge u_p,v_1\wedge\ldots\wedge v_p\rangle\  =\ \det [g_x(u_i,v_j)].
\]
In particular, if $\{e_r, r=1,\ldots,n\}$ is an orthonormal basis of $T_xM$, then
$\{e_{i_1}\wedge\ldots\wedge e_{i_p},1\leq i_1<\ldots<i_p\leq n\}$ is an orthonormal basis for $\Lambda^pT_xM$.
Furthermore, by means of the inner product,
we identify $\Lambda^pT_xM$ with the space $\Lambda^pT^*_xM$ of $p$-forms on $T_xM$.

The curvature tensor $R$ induces by \eqref{pairs} a symmetric linear operator
\[
\mathcal{R}:\Lambda^2 T_xM\rightarrow \Lambda^2T_xM\qquad \langle\mathcal{R}(X\wedge Y),Z\wedge V\rangle =-g(R(X,Y)Z,V).
\]
The sign $-$ is due to our curvature convention, so that positive curvature operator $\mathcal{R}$ implies positive
sectional curvature
\[
\operatorname{sec}(X,Y)=R(X,Y,Y,X).
\]
Indeed, identifying a $2$-plane $\sigma\subset T_xM$ with the $2$-vector $X\wedge Y\in\Lambda^2 T_xM$,
 where $X,Y$ is an orthonormal basis of $T_xM$, the sectional curvature is
$\operatorname{sec}(\sigma)=\langle\mathcal{R}(\sigma),\sigma\rangle$.

Any $4$-form $\omega$ can be regarded as a symmetric operator
\[
\omega:\Lambda^2 T_xM\rightarrow \Lambda^2T_xM\qquad \langle\omega(\alpha),\beta\rangle=\langle\omega,\alpha\wedge\beta\rangle.
\]
In fact, the space of all symmetric linear operators splits as $S(\Lambda^2T_xM) = \ker \frak{b}\oplus \Lambda^4T_xM$,
where $\mathfrak{b}$ denotes the Bianchi map
\[
\mathfrak{b}(\Omega)(X,Y,Z,V):=\Omega(X,Y,Z,V)+\Omega(Y,Z,X,V)+\Omega(Z,X,Y,V).
\]
Then, $\ker \frak{b}$ is the space of algebraic curvature operators\footnote{Note that the curvature operator of a connection with torsion is \emph{not} an algebraic curvature operator by this definition, compare \eqref{bianchi}.}, i.e. operators satisfying the first Bianchi identity \eqref{bianchi} for vanishing torsion.
\begin{df}\label{Def-S_TG_T}
%----------------------------
We will denote by ${\Sc}_T:\Lambda^2M\to\Lambda^2M$ the symmetric operator associated to the $4$-form $\sigma_T$, i.e.
\begin{equation}\label{S_T}
 \langle\Sc_T(X\wedge Y),Z\wedge V\rangle:=\sigma_T(X,Y,Z,V)=\frac12 dT(X,Y,Z,V).
\end{equation}
We will also consider the $(0,4)$-tensor field  $G_T$ and the symmetric operator $\G_T:\Lambda^2M\to\Lambda^2M$ defined by
\[
\langle{\mathcal G}_T(X\wedge Y),Z\wedge V\rangle=G_T(X,Y,Z,V):=g(T(X,Y),T(Z,V)).
\]
\end{df}
Owing to \eqref{Rg-R}, we have
\begin{equation}\label{Rg-Rshort}
\Rc^g=\Rc+\frac14 \G_T+\frac14 \Sc_T.
\end{equation}
\begin{df}
%------------
A Riemannian manifold $(M, g)$ is said to have \emph{strongly positive
curvature} (resp.\, \emph{strongly non-negative curvature}) if there exists a $4$-form $\omega$ such that
$\mathcal{R}^g+\omega$ is positive-definite (resp.\,non-negative) at every point $x\in M$ \cite{Thorpe71,StrPosCurv}.
\end{df}

Such a notion is justified by the fact that for every $2$-plane $\sigma$, being $\langle\omega(\sigma),\sigma\rangle=0$, one has
$\operatorname{sec}(\sigma)=\langle(\mathcal{R}^g+\omega)(\sigma),\sigma\rangle$,
so that strongly positive curvature implies positive sectional curvature. In fact this is an intermediate notion
between positive definiteness of the Riemannian curvature ($\mathcal{R}^g> 0$) and positive sectional curvature.
%
%----------------------------------------------------------------------------------------
\subsection{Review of \texorpdfstring{$3$-$(\alpha,\delta)$}{3-(a,d)}-Sasaki manifolds and their basic properties}
%----------------------------------------------------------------------------------------
%
We now want to focus on the situation at hand. That is a $3$-$(\alpha,\delta)$-Sasaki manifold and its canonical connection $\nabla$. Let us recall the central definitions and key properties for later reference.

%--------------------------------------------------------------
An \emph{almost contact metric structure} on a $(2n+1)$-dimensional differentiable manifold
$M$ is a quadruple $(\varphi,\xi,\eta,g)$, where $\varphi$ is a $(1,1)$-tensor
field, $\xi$ a vector field, called the \emph{characteristic} or \emph{Reeb vector field}, $\eta$ a $1$-form, $g$ a Riemannian metric, such that
\begin{gather*} %\label{defi:cont}
\varphi^2=-I+\eta\otimes \xi,\quad  \eta(\xi)=1,\quad \varphi (\xi) =0,\quad \eta \circ \varphi =0,\\
g(\varphi X,\varphi Y)=g(X,Y)-\eta (X) \eta(Y)\quad \forall X,Y\in{\frak X}(M).
\end{gather*}
It follows that $\eta=g(\cdot,\xi)$, and $\varphi$ induces a complex structure on the $2n$-dimensional distribution given by
$\mathrm{Im}(\varphi)=\ker\eta=\langle \xi\rangle^\perp$. The fundamental $2$-form associated to the structure is defined by
$\Phi(X,Y)=g(X,\varphi Y)$.
The almost contact metric structure is said to be
\emph{normal} if $ N_\varphi\coloneqq[\varphi,\varphi]+d\eta\otimes\xi$ vanishes,
where $[\varphi,\varphi]$ is the Nijenhuis torsion of $\varphi$ \cite{BLAIR}.
An \emph{$\alpha$-Sasaki manifold} is defined as a normal almost contact metric manifold
such that
$
d\eta\, =\, 2\alpha\Phi$, $\alpha\in\R^*.
$
For $\alpha=1$, this is a \emph{Sasaki manifold}. %The $1$-form
%$\eta$ of an $\alpha$-Sasaki structure is a \emph{contact form}, in the sense that
%$\eta\wedge (d\eta)^n\ne 0$ everywhere on $M$. The Reeb vector field is always Killing.

%
An  \emph{almost $3$-contact metric manifold}  is a differentiable  manifold $M$ of dimension
$4n+3$ endowed with three almost contact metric structures $(\varphi_i,\xi_i,\eta_i,g)$,
$i=1,2,3$, sharing the same Riemannian metric $g$, and satisfying the following compatibility
relations
%
%\begin{equation}\label{3-sasaki}
\bdm
\varphi_k=\varphi_i\varphi_j-\eta_j\otimes\xi_i=-\varphi_j\varphi_i+\eta_i\otimes\xi_j,\quad
\xi_k=\varphi_i\xi_j=-\varphi_j\xi_i, \quad
\eta_k=\eta_i\circ\varphi_j=-\eta_j\circ\varphi_i
\edm
%\end{equation}
%
for any even permutation $(ijk)$ of $(123)$ \cite{BLAIR}.
The tangent bundle of $M$ splits into the orthogonal sum $TM=\H\oplus\V$, where $\H$ and
$\V$ are respectively the \emph{horizontal} and the \emph{vertical} distribution, defined by
\[
\H \, \coloneqq\, \bigcap_{i=1}^{3}\ker\eta_i,\qquad
\V\, \coloneqq\, \langle\xi_1,\xi_2,\xi_3\rangle.
\]
In particular $\H$ has rank $4n$ and the three Reeb vector
fields $\xi_1,\xi_2,\xi_3$ are orthonormal.
%The structure group of the tangent
%bundle is in fact reducible to $\mathrm{Sp}(n)\times \{1\}$ \cite{Kuo70}; this implies in
%particular that each almost $3$-contact metric manifold is spin.
The manifold is said to  be \emph{hypernormal} if each  almost contact metric structure
$(\varphi_i,\xi_i,\eta_i,g)$ is normal. If the three structures are $\alpha$-Sasaki, $M$ is called a $3$-$\alpha$-Sasaki manifold, $3$-Sasaki if $\alpha=1$. As a  comprehensive introduction
to Sasaki and $3$-Sasaki geometry, we refer to \cite{Boyer&Galicki}.
We denote an almost $3$-contact metric manifold by $(M,\varphi_i,\xi_i,\eta_i, g)$, understanding
that the index is running from $1$ to $3$.

The class of $3$-$(\alpha,\delta)$-Sasaki manifolds was introduced in \cite{AgrDil} as a generalization of $3$-$\alpha$-Sasaki manifolds, and further investigated in \cite{ADS20}.
%

%-------------------------------
\begin{definition}
An almost $3$-contact metric manifold $(M,\varphi_i,\xi_i,\eta_i,g)$ is called a
\emph{$3$-$(\alpha,\delta)$-Sasaki manifold} if it satisfies
\begin{equation}\label{differential_eta}
d\eta_i=2\alpha\Phi_i+2(\alpha-\delta)\eta_j\wedge\eta_k
\end{equation}
for every even permutation $(ijk)$ of $(123)$, where $\alpha\neq 0$ and
$\delta$ are real constants.
A $3$-$(\alpha,\delta)$-Sasaki manifold is called \emph{degenerate} if
$\delta=0$ and \emph{non-degenerate} otherwise.
Non-degenerate $3$-$(\alpha,\delta)$-Sasaki manifolds are called
\emph{positive} or \emph{negative}, depending on whether  $\alpha\delta>0$ or $\alpha\delta<0$.
\end{definition}
The distinction into degenerate, positive, and negative
$3$-$(\alpha,\delta)$-Sasaki manifolds stems from their behaviour under
a special type of deformations of the structure, called \emph{$\H$-homothetic deformations}, which turn out to preserve the three classes \cite[Section 2.3]{AgrDil}.
%
%%\begin{equation}\label{deformation}
%\bdm
%\eta_i'=c\eta_i,\quad \xi_i'=\frac{1}{c}\xi_i,\quad \varphi_i'=\varphi_i,\quad
%g'=ag+b\sum_{i=1}^3\eta_i\otimes\eta_i \qquad \text{with }a>0 \text{ and }c^2=a+b>0.
%%\end{equation}
%\edm
%%
%The deformed structure $(\varphi',\xi_i',\eta',g')$ turns out to be $3$-$(\alpha',\delta')$-Sasaki with
%$\alpha'=\alpha c/a$, $\delta'=\delta/c$.
%In particular, $\H$-homothetic deformations preserve the class of degenerate
%$3$-$(\alpha,\delta)$-Sasaki manifolds. In the non-degenerate case the sign of the
%product $\alpha\delta$ is also preserved.

%
We recall some basic properties of $3$-$(\alpha,\delta)$-Sasaki manifolds. Any $3$-$(\alpha,\delta)$-Sasaki manifold is hypernormal. Hence, for $\alpha=\delta$ one has a
$3$-$\alpha$-Sasaki manifold. Each Reeb vector field $\xi_i$ is Killing and it is an infinitesimal
automorphism of the horizontal distribution $\H$, i.e. $d\eta_i(X,\xi_j)=0$ for every
$X\in\H $ and $i,j=1,2,3$. The vertical distribution $\V$ is integrable with totally
geodesic leaves.
In particular, the commutators of the Reeb vector fields are purely vertical and for every even permutation $(ijk)$ of $(123)$ they are given by
\bdm
[\xi_i,\xi_j]=2\delta\xi_k.
\edm
Meanwhile, for any two horizontal vector fields $X,Y$, the vertical part of commutators is given by
\begin{equation}\label{VComm}
[X,Y]_\mathcal{V}=-2\alpha\sum_{i=1}^3\Phi_i(X,Y)\xi_i.
\end{equation}

Any $3$-$(\alpha,\delta)$-Sasaki manifolds is a
\emph{canonical almost $3$-contact metric manifold}, in the sense of the definition given in \cite{AgrDil}, which is equivalent
to the existence of a \emph{canonical connection}.
%We recall here some basic facts about connections with totally skew-symmetric torsion---we
%refer to \cite{Ag} for further details. A metric
%connection $\nabla$ with torsion $T$ on a Riemannian manifold $(M,g)$ is said to have
%\emph{totally skew-symmetric torsion},
%or \emph{skew torsion} for short, if the $(0,3)$-tensor field $T$ defined by
%
%\[T(X,Y,Z)=g(T(X,Y),Z)\]
%
%is a $3$-form. The relation between $\nabla$ and the Levi-Civita connection
%$\nabla^g$ is then given by
%
%\bdm
%\nabla_XY=\nabla^g_XY+\frac{1}{2}T(X,Y).
%\edm
%
%It is well-known that any Sasaki manifold $(M,\varphi,\xi,\eta,g)$ admits a \emph{characteristic connection},
%i.\,e.~a unique metric connection $\nabla$ with skew torsion such that $\nabla\eta=\nabla\varphi=0$. Its torsion
%is given by $T=\eta\wedge d\eta$ \cite{FrIv}. As a
%consequence, a $3$-Sasaki manifold $(M,\varphi_i,\xi_i,\eta_i,g)$ cannot admit any metric connection
%with skew torsion such that $\nabla\eta_i=\nabla\varphi_i=0$ for every $i=1,2,3$. By relaxing
%the requirement on the parallelism of the structure tensor fields in a suitable way,
%one can define a large class of almost $3$-contact metric
%manifolds, called \emph{canonical}, including $3$-$(\alpha,\delta)$-Sasaki manifolds, and thus
%$3$-Sasaki manifolds.
The canonical connection of a  $3$-$(\alpha,\delta)$-Sasaki manifold $(M,\varphi_i,\xi_i,\eta_i,g)$ is the unique metric connection $\nabla$ with skew torsion such that
\begin{equation}\label{canonical}
\nabla_X\varphi_i\, =\, \beta(\eta_k(X)\varphi _j -\eta_j(X)\varphi _k) \quad
\forall X\in{\frak X}(M)
\end{equation}
for every even permutation $(ijk)$ of $(123)$, where $\beta=2(\delta-2\alpha)$. The
covariant derivatives  of the other structure tensor fields are given by
\[\nabla_X\xi_i=\beta(\eta_k(X)\xi_j-\eta_j(X)\xi_k),\qquad\nabla_X\eta_i=\beta(\eta_k(X)\eta_j-\eta_j(X)\eta_k).\]
If $\delta=2\alpha$, then $\beta=0$ and the canonical connection parallelizes all the structure tensor fields.
Any $3$-$(\alpha,\delta)$-Sasaki manifold with $\delta=2\alpha$, which is a positive $3$-$(\alpha,\delta)$-Sasaki manifold, is called \emph{parallel}.

The canonical connection plays a central role in the description of the transverse geometry defined by the vertical foliation:

\begin{theorem}[{\cite[Prop. 2.1.1, Theorem 2.2.1, Theorem 2.2.2]{ADS20}}]\label{theorem-canonical subm}
Every $3$-$(\alpha,\delta)$-Sasaki manifold $M$ gives rise to a locally defined Riemannian submersion $\pi\colon (M,g)\to (N,g_N)$ with fibers spanned by $\mathcal{V}$ and
\[
\nabla^{g_N}_XY=\pi_*(\nabla_{\overline{X}}\overline{Y}).
\]
The base space $N$ is equipped with a quaternion K\"ahler structure locally defined by $\check{\varphi}_i=\pi_*\circ\varphi_i\circ s$, $i=1,2,3$, where $s\colon N\to M$ is an arbitrary section of $\pi$. The scalar curvature of $N$ is $16n(n+2)\alpha\delta$.
\end{theorem}
Here and in the following $\overline{X}\in TM$ denotes the horizontal lift of a vector field $X\in TN$ under the Riemannian submersion $\pi\colon M\to N$. We further denote the Levi-Civita connection on $(N,g_N)$ by $\nabla^{g_N}$ and analogously for its associated tensors, e.g.\ the curvature tensor $R^{g_N}$.

From the above Theorem it follows that any $3$-$(\alpha,\delta)$-Sasaki manifold locally fibers over a quaternionic K\"ahler manifold of positive or negative scalar curvature if either $\alpha\delta>0$ or $\alpha\delta<0$ respectively, or over a hyper-K\"ahler manifold in the degenerate case.

The Riemannian Ricci tensor of a $3$-$(\alpha,\delta)$-Sasaki manifold has been computed in \cite[Proposition 2.3.3]{AgrDil}:
\begin{equation}\label{riccig_complete}
Ric^g(X,Y)=2\alpha\{2\delta(n+2)-3\alpha\}g(X,Y)+2(\alpha-\delta)\{(2n+3)\alpha-\delta\}
\sum_{i=1}^3\eta_i(X)\eta_i(Y)
\end{equation}
implying that the manifold is Riemannian Einstein if and only if $\delta=\alpha$ or $\delta=\alpha(2n+3)$.

Finally, we recall some properties for the torsion of the canonical connection.
The torsion $T$ of the canonical connection of a $3$-$(\alpha,\delta)$-Sasaki manifold is given by
%\begin{equation}
%T\ =\ \sum_{i=1}^3\eta_i\wedge d\eta_i+8(\delta-\alpha)\,\eta_{123},
%\end{equation}
%or equivalently,
\begin{equation}\label{torsion01}
T\ =\ 2\alpha\sum_{i=1}^3\eta_i\wedge\Phi_i-2(\alpha-\delta)\eta_{123}\ =\
2\alpha \sum_{i=1}^3\eta_i\wedge \Phi^{\mathcal H}_i+2(\delta-4 \alpha)\,\eta_{123},
\end{equation}
where $\Phi^{\mathcal H}_i=\Phi_i+\eta_{jk}\in\Lambda^2({\mathcal H})$ is the horizontal part of the
fundamental $2$-form $\Phi_i$. Here we put $\eta_{jk}\coloneqq\eta_j\wedge\eta_k$ and $\eta_{123}\coloneqq\eta_1\wedge\eta_2\wedge\eta_3$.
In particular, for every $X,Y\in\frak{X}(M)$,
\begin{equation}\label{torsion02}
T(X,Y)=2\alpha\sum_{i=1}^3\{\eta_i(Y)\varphi_iX-\eta_i(X)\varphi_iY+\Phi_i(X,Y)\xi_i\}-2(\alpha-\delta)
\cyclic{i,j,k}\eta_{ij}(X,Y)\xi_k.
\end{equation}
The symbol $\cyclic{i,j,k}$ means the sum over all even permutations of
$(123)$.
The torsion of the canonical connection satisfies $\nabla T=0$ and
\begin{align}
\begin{split}dT & =  4\alpha^2\sum_{i=1}^3\Phi_i\wedge\Phi_i + 8\alpha(\delta-\alpha)
\cyclic{i,j,k}\Phi_i\wedge\eta_{jk}\label{dT}\\
& =  4\alpha^2\sum_{i=1}^3\Phi^\H_i\wedge\Phi^\H_i +
8\alpha(\delta-2 \alpha) \cyclic{i,j,k}\Phi^\H_i\wedge\eta_{jk}.
\end{split}
\end{align}
%
%\input{preliminaries}
%!TeX root=ADS3adCurv-bb.tex

%-----------------------------------------------------------------------------------------
\section{The Canonical Curvature Operator}\label{section2}
%-----------------------------------------------------------------------------------------
%
%
\subsection{The canonical curvature and the canonical submersion}
The canonical curvature is particularly well behaved on the defining tensors of a $3$-$(\alpha,\delta)$-Sasaki manifold. We will make use of this to compute directly related curvature identities in the following two propositions. These, in turn, allowed us to prove the existence of the canonical submersion in \cite{ADS20}.
\begin{prop}
%---------------
Let $(M,\varphi_i,\xi_i,\eta_i,g)$ be a $3$-$(\alpha,\delta)$-Sasaki manifold. Let $\nabla$ be the canonical
connection and $R$ the curvature tensor of $\nabla$. Then, the following equations hold:
\begin{equation}
\begin{split}\label{curvature_phi}
&R(X,Y)\varphi_iZ-\varphi_i R(X,Y)Z=2\alpha\beta\{\Phi_k(X,Y)\varphi_jZ-\Phi_j(X,Y)\varphi_kZ\}\\
&\qquad\qquad\qquad\qquad\qquad\qquad-2\alpha\beta\{(\eta_i\wedge\eta_j)(X,Y)\varphi_jZ-(\eta_k\wedge\eta_i)(X,Y)\varphi_kZ\},
\end{split}
\end{equation}
\begin{equation}
\begin{split}\label{curvature_xi}
R(X,Y)\xi_i&=2\alpha\beta\{\Phi_k(X,Y)\xi_j-\Phi_j(X,Y)\xi_k\}\\
&\quad-2\alpha\beta\{(\eta_i\wedge\eta_j)(X,Y)\xi_j-(\eta_k\wedge\eta_i)(X,Y)\xi_k\},
\end{split}
\end{equation}
where $X,Y,Z\in\frak{X}(M)$ and $(ijk)$ is an even permutation of $(123)$.
\end{prop}
\begin{proof}
%-----------------
Applying the Ricci identity, \eqref{canonical} and \eqref{differential_eta}, we have
\begin{align*}
&R(X,Y)\varphi_iZ-\varphi_i R(X,Y)Z\\
&=(\nabla_X(\nabla_Y\varphi_i))Z-(\nabla_Y(\nabla_X\varphi_i))Z-(\nabla_{[X,Y]}\varphi_i)Z\\
&=\beta\{X(\eta_k(Y))\varphi_jZ+\eta_k(Y)(\nabla_X\varphi_j)Z-X(\eta_j(Y))\varphi_kZ-\eta_j(Y)(\nabla_X\varphi_k)Z\}\\
&\quad-\beta\{Y(\eta_k(X))\varphi_jZ+\eta_k(X)(\nabla_Y\varphi_j)Z-Y(\eta_j(X))\varphi_kZ-\eta_j(X)(\nabla_Y\varphi_k)Z\}\\
&\quad-\beta\{\eta_k([X,Y])\varphi_jZ-\eta_j([X,Y])\varphi_kZ\}\\
&=\beta\{d\eta_k(X,Y)\varphi_jZ-d\eta_j(X,Y)\varphi_kZ\}\\
&\quad+\beta^2\{\eta_k(Y)(\eta_i(X)\varphi_kZ-\eta_k(X)\varphi_iZ)-\eta_j(Y)(\eta_j(X)\varphi_iZ-\eta_i(X)\varphi_jZ)\}\\
&\quad-\beta^2\{\eta_k(X)(\eta_i(Y)\varphi_kZ-\eta_k(Y)\varphi_iZ)-\eta_j(X)(\eta_j(Y)\varphi_iZ-\eta_i(Y)\varphi_jZ)\}\\
&=2\alpha\beta\{\Phi_k(X,Y)\varphi_jZ-\Phi_j(X,Y)\varphi_kZ\}\\
&\quad+\{2\beta(\alpha-\delta)+\beta^2\}\{(\eta_i\wedge\eta_j)(X,Y)\varphi_jZ-(\eta_k\wedge\eta_i)(X,Y)\varphi_kZ\},
\end{align*}
which gives \eqref{curvature_phi}, since $\beta=2(\delta-2\alpha)$. We obtain \eqref{curvature_xi} by setting $Z=\xi_i$ in \eqref{curvature_phi} and applying $\varphi_i$ from the left. 
%$2\alpha\beta=4\alpha(\delta-2\alpha)$ and
%\[2\beta(\alpha-\delta)+\beta^2=\beta(2\alpha-2\delta+2\delta-4\alpha)=-2\alpha\beta.\]
%The proof of \eqref{curvature_xi} proceeds analogously, using $\nabla_X\xi_i=\beta(\eta_k(X)\xi_j-\eta_j(X)\xi_k)$.
%Indeed,
%\begin{align*}
%R(X,Y)\xi_i&=\nabla_X\nabla_Y\xi_i-\nabla_Y\nabla_X\xi_i-\nabla_{[X,Y]}\xi_i\\
%&=\beta\{X(\eta_k(Y))\xi_j+\eta_k(Y)\nabla_X\xi_j-X(\eta_j(Y))\xi_k-\eta_j(Y)\nabla_X\xi_k\}\\
%&\quad-\beta\{Y(\eta_k(X))\xi_j+\eta_k(X)\nabla_Y\xi_j-Y(\eta_j(X))\xi_k-\eta_j(X)\nabla_Y\xi_k\}\\
%&\quad-\beta\{\eta_k([X,Y])\xi_j-\eta_j([X,Y])\xi_k\}\\
%&=\beta\{d\eta_k(X,Y)\xi_j-d\eta_j(X,Y)\xi_k\}\\
%&\quad+\beta^2\{\eta_k(Y)(\eta_i(X)\xi_k-\eta_k(X)\xi_i)-\eta_j(Y)(\eta_j(X)\xi_i-\eta_i(X)\xi_j)\}\\
%&\quad-\beta^2\{\eta_k(X)(\eta_i(Y)\xi_k-\eta_k(Y)\xi_i)-\eta_j(X)(\eta_j(Y)\xi_i-\eta_i(Y)\xi_j)\}\\
%&=2\alpha\beta\{\Phi_k(X,Y)\xi_j-\Phi_j(X,Y)\xi_k\}\\
%&\quad+\{2\beta(\alpha-\delta)+\beta^2\}\{(\eta_i\wedge\eta_j)(X,Y)\xi_j-(\eta_k\wedge\eta_i)(X,Y)\xi_k\},
%\end{align*}
%which gives \eqref{curvature_xi}.
\end{proof}
%
%----------------------------------------------------------------------------------------------------
\begin{prop}
%-------------------------------------------
The curvature tensor $R$ of the canonical connection of a $3$-$(\alpha,\delta)$-Sasaki manifold
satisfies for any $X,Y,Z\in\H $ and $i,j,k,l=1,2,3$ the identities
\begin{align}
\hspace{2.15cm}R(X,\xi_i,Y,\xi_j)&=R(X,Y,Z,\xi_i)=R(\xi_i,\xi_j,\xi_k,X)=0,\label{0Curv}\\
R(\xi_i,\xi_j,\xi_k,\xi_l)&=-4\alpha\beta (\delta_{ik}\delta_{jl}-\delta_{il}\delta_{jk}),\label{1Curv}
\end{align}
and for an even permutation $(ijk)$ of $(123)$
\begin{align}
R(\xi_i,\xi_j,X,Y)&=2\alpha\beta \Phi_k(X,Y),\hspace{4cm}\label{offdiag}\\
R(X,Y,Z,\varphi_iZ)+R(X,Y,\varphi_jZ,\varphi_kZ)&=2\alpha\beta\Phi_i(X,Y)\|Z\|^2.\label{2Curv}
\end{align}
%where $\pm$ refers to an even, respectively odd, permutation $(ijk)$ of $(123)$.
%
\end{prop}
\begin{proof}
Considering the symmetries of $R$ we immediately obtain the first three expressions from equation \eqref{curvature_xi}. Using \eqref{curvature_phi} for $\varphi_j$ we obtain
\begin{align*}
R(X,Y,\varphi_j Z,\varphi_k Z)&= g(\varphi_jR(X,Y)Z,\varphi_kZ)\\
&\qquad + 2\alpha\beta(\Phi_i(X,Y)g(\varphi_kZ,\varphi_k Z)-\Phi_k(X,Y)g(\varphi_iZ,\varphi_k Z))\\
&=-R(X,Y,Z,\varphi_i Z) +2\alpha\beta \Phi_i(X,Y)\|Z\|^2.\qedhere
\end{align*}
\end{proof}
\begin{rem}
The identities \eqref{0Curv}, \eqref{1Curv}, \eqref{offdiag}, \eqref{2Curv} are used to prove the canonical submersion in \cite{ADS20}.
\end{rem}
Considering now the canonical submersion $\pi\colon M\to N$  defined in Theorem \ref{theorem-canonical subm}, in the next Theorem we will relate the missing purely horizontal part of the canonical curvature tensor to the curvature of the quaternionic K\"ahler base space $N$. We recall a computational lemma from \cite{ADS20}.
%
% We recall the canonical submersion from \cite{ADS20}. In the following we denote by $\overline{X}\in TM$ the horizontal lift of a vector filed $X\in TN$ under a Riemannian submersion $\pi\colon M\to N$.
%%
%\begin{theorem}[{\cite[Prop. 2.1.1 and Theorem 2.2.1]{ADS20}}]
%Every $3$-$(\alpha,\delta)$-Sasaki manifolds $M$ gives rise to a locally defined Riemannian submersion $\pi\colon M\to N$ with fibers spanned by $\mathcal{V}$ and
%\[
%\nabla^{g_N}_XY=\pi_*(\nabla_{\overline{X}}\overline{Y}).
%\]
%The base space $N$ is equipped with a quaternionic K\"ahler structure locally defined by $\check{\varphi}_i=\pi\circ\varphi_i\circ s$, $i=1,2,3$, where $s\colon N\to M$ is an arbitrary section of $\pi$.
%\end{theorem}
%%
%
\begin{lemma}[{\cite[Lemma 2.2.1]{ADS20}}]\label{Vnabla}
For any vertical vector field $X\in\mathcal{V}$ and for any basic vector field $Y\in\mathcal{H}$ we have
\[
(\nabla_{X}Y)_{\mathcal{H}}=-2\alpha\sum_{i=1}^3\eta_i(X)\varphi_i Y.
\]
\end{lemma}
%
%The canonical submersion $\pi\colon M\to N$ now allows us to relate the missing purely horizontal part of the canonical curvature tensor to the curvature of the base.
%
\begin{theorem}\label{HCurv}
The canonical curvature on $\Lambda^2\mathcal{H}\otimes\Lambda^2\mathcal{H}$ is given by
\begin{align*}
R(\overline{X},\overline{Y},\overline{Z},\overline{V})=R^{g_N}(X,Y,Z,V)+4\alpha^2\sum_{i=1}^3\Phi_i(\overline{X},\overline{Y})\Phi_i(\overline{Z},\overline{V}),
\end{align*}
where $X,Y,Z,V\in TN$ with horizontal lifts $\overline{X},\overline{Y},\overline{Z},\overline{V}\in \mathcal{H}$.
\end{theorem}
\begin{proof}
We note that by $\nabla^{g_N}_XY=\pi_*(\nabla_{\overline{X}}\overline{Y})$ the vector field $\nabla_{\overline{X}}\overline{Y}\in \mathcal{H}$ is $\pi$-related to $\nabla^{g_N}_XY$ and thus $\nabla_{\overline{X}}\overline{Y}=\overline{\pi_*(\nabla_{\overline{X}}\overline{Y})}$. We obtain
\begin{align*}
g_N(\nabla^{g_N}_X\nabla^{g_N}_Y Z,V)&=g_N(\nabla^{g_N}_X\pi_*(\nabla_{\overline{Y}}\overline{Z}),V)
=g(\nabla_{\overline{X}}\overline{\pi_*(\nabla_{\overline{Y}}\overline{Z})},\overline{V})
=g(\nabla_{\overline{X}}\nabla_{\overline{Y}}\overline{Z},\overline{V}),\\
g_N(\nabla^{g_N}_{[X,Y]}Z,V)&=g_N(\pi_*\nabla_{\overline{[X,Y]}}\overline{Z},V)=g(\nabla_{[\overline{X},\overline{Y}]}\overline{Z}-\nabla_{[\overline{X},\overline{Y}]_{\mathcal{V}}}\overline{Z},\overline{V})\\
&=g(\nabla_{[\overline{X},\overline{Y}]}\overline{Z},\overline{V})+4\alpha^2\sum_{i=1}^{3}\Phi_i(\overline{X},\overline{Y})\Phi_i(\overline{Z},\overline{V})
\end{align*}
where we have used \eqref{VComm} and \autoref{Vnabla}. Plugging these identities into the curvature we find
\begin{align*}
R^{g_N}(X,Y,Z,V)&=g_N(\nabla^{g_N}_X\nabla^{g_N}_YZ-\nabla^{g_N}_Y\nabla^{g_N}_XZ-\nabla^{g_N}_{[X,Y]}Z,V)\\
&=g(\nabla_{\overline{X}}\nabla_{\overline{Y}}\overline{Z}-\nabla_{\overline{Y}}\nabla_{\overline{X}}\overline{Z}-\nabla_{[\overline{X},\overline{Y}]}\overline{Z},\overline{V})-4\alpha^2\sum_{i=1}^{3}\Phi_i(\overline{X},\overline{Y})\Phi_i(\overline{Z},\overline{V})\\
&=R(\overline{X},\overline{Y},\overline{Z},\overline{V})-4\alpha^2\sum_{i=1}^{3}\Phi_i(\overline{X},\overline{Y})\Phi_i(\overline{Z},\overline{V}).\qedhere
\end{align*}
\end{proof}
\subsection{Decomposition of the Canonical Curvature Operator}
%
%To investigate eigenvalues and definiteness of curvature operators the following proposition decomposes the canonical curvature into a part coming from the base and a part encoding central features of the $3$-$(\alpha,\delta)$-Sasaki structure.
We now want to look at the canonical curvature as a curvature operator and consider its eigenvalues and definiteness. Recall that the canonical curvature operator $\Rc\colon \Lambda^2M\to\Lambda^2 M$ defines a symmetric operator. Rewriting \eqref{0Curv} as operator identities we obtain
\[
\langle \Rc(X\wedge \xi_i),Y\wedge\xi_j\rangle=\langle \Rc(X\wedge Y),Z\wedge \xi_i\rangle=\langle \Rc(\xi_i\wedge\xi_j),\xi_k\wedge X\rangle=0
\]
showing that the canonical curvature operator vanishes on $\mathcal{V}\wedge \mathcal{H}$. Thus, it can be considered as a symmetric operator $\Rc\colon \Lambda^2\mathcal{V}\oplus\Lambda^2\mathcal{H}\to \Lambda^2\mathcal{V}\oplus\Lambda^2\mathcal{H}$. It does not restrict to the individual summands but we can accomplish a more nuanced decomposition.
\begin{rem}
From here on out we will freely identify $TM$ and $T^*M$ as well as their exterior products. In particular, we write $\xi_{jk}\coloneqq \xi_j\wedge\xi_k=\eta_{jk}=\eta_j\wedge\eta_k$.
\end{rem}
% In the following, we put $\xi_{jk}:=\xi_j\wedge\xi_k$.
%
\begin{prop}\label{decompR}
%-----------------------------------------------
The curvature operator $\Rc$ can be decomposed as
\[
\Rc=\alpha\beta \Rc_{\perp}+\Rcpar
\]
where $\Rcperp$ is defined by
\begin{align}\label{Rcperp}
\Rcperp\coloneqq \cyclic{i,j,k}(\Phi_i-\xi_{jk})\otimes(\Phi_i-\xi_{jk})
\end{align}
and $\Rcpar$ is trivial outside of the horizontal part, i.e.~$\Rcpar|_{(\Lambda^2\mathcal{H})^\perp}=0$.
\end{prop}
\begin{proof}
 Equations \eqref{1Curv} and \eqref{offdiag} in terms of the curvature operator read
\begin{align}
\langle \Rc(\xi_i\wedge\xi_j),\xi_k\wedge \xi_l\rangle &=4\alpha\beta\sum_{\mu=1}^3\Phi_\mu(\xi_i,\xi_j)\Phi_\mu(\xi_k,\xi_l),\label{Rc2}\\
\langle \Rc(\xi_i\wedge \xi_j),X\wedge Y\rangle&=2\alpha\beta\sum_{\mu=1}^3\Phi_\mu(\xi_i,\xi_j)\Phi_\mu(X,Y).\label{Rc3}
\end{align}
%and the horizontal part is given by \autoref{HCurv} as
%\[
%\langle\Rc(\overline{X}\wedge \overline{Y}),\overline{Z}\wedge \overline{V}\rangle=\langle\Rc^{g_N}(X\wedge Y),Z\wedge V\rangle-4\alpha^2\sum_{i=1}^3\Phi_i(\overline{X},\overline{Y})\Phi_i(\overline{Z},\overline{V}).
%\]
We observe that identities \eqref{Rc2} and \eqref{Rc3} are of the form $C\sum_{i=1}^3\Phi_i\otimes \Phi_i$, where the coefficient $C$ is $4\alpha\beta$ or $2\alpha\beta$. The comparison with%Using the symmetry of $\Rc$ and setting the coefficient $1$ in the purely horizontal part we obtain
\begin{align*}
\Rcperp&\coloneqq \cyclic{i,j,k}(\Phi_i-\xi_{jk})\otimes(\Phi_i-\xi_{jk})\\
&=\sum_{i=1}^3\big(4(\Phi_i|_{\Lambda^2\V}\otimes\Phi_i|_{\Lambda^2\V})+2(\Phi_i|_{\Lambda^2\H}\otimes\Phi_i|_{\Lambda^2\V})+2(\Phi_i|_{\Lambda^2\V}\otimes\Phi_i|_{\Lambda^2\H})+(\Phi_i|_{\Lambda^2\H}\otimes\Phi_i|_{\Lambda^2\H})\big).
\end{align*}
shows that $\Rcpar\coloneqq \Rc-\alpha\beta\Rcperp$ is trivial outside $\Lambda^2\H\rightarrow\Lambda^2\H$.
\end{proof}
The notation $\Rcpar$ is justified by the fact that in the parallel case ($\beta=0$) we have $\Rc=\Rcpar$.
Taking into account the canonical submersion $\pi\colon M\to N$, we may consider the Riemannian curvature operator $\Rc^{g_N}$ on the base $N$ as a curvature operator $\Lambda^2\H\to\Lambda^2\H$ via the horizontal lift. From \autoref{HCurv}, we have
\[\Rc|_{\Lambda^2\H\otimes\Lambda^2\H}=\Rc^{g_N}-4\alpha^2\sum_{\mu=1}^3\Phi^{\H}_\mu\otimes\Phi^{\H}_\mu.\]
Note the sign change due to our convention of sign in $\Rc$ compared to $R$. Comparing with the definition of $\Rcpar$ in \autoref{decompR} and expanding $\beta=2(\delta-2\alpha)$ yields
%We then compare $\Rcpar$ with the Riemannian curvature operator to obtain
\begin{equation}\label{RgNdecomp}
\Rcpar=\Rc^{g_N}-2\alpha\delta\sum_{\mu=1}^3\Phi^{\H}_\mu\otimes\Phi^{\H}_\mu.
\end{equation}

\begin{rem}
Recall that the curvature operator of qK spaces is given by $\mathcal{R}^{g_N}=\nu \mathcal{R}_0+\mathcal{R}_1$, where $\nu=4\alpha\delta$ is the reduced scalar curvature, $\mathcal{R}_1$ is a curvature operator of hyper-K\"ahler type and
\[
\mathcal{R}_0=\frac 18\left(g\owedge g+\sum_{\mu=1}^3\Phi_\mu^\mathcal{H}\owedge\Phi_\mu^\mathcal{H}+4\Phi_\mu^{\mathcal{H}}\otimes\Phi_\mu^{\mathcal{H}}\right)
\]
is the curvature operator of $\mathbb{H}P^n$ \cite[Table 1 and 2]{Ale68}. Here $\owedge$ denotes the Kulkarni-Nomizu product viewed as an operator. A curvature operator is said to be of hyper-K\"ahler type if it is Ricci-flat and commutes with the quaternionic structure. Combining this with \eqref{RgNdecomp} we find
\[
\mathcal{R}_{\mathrm{par}}=\frac{\alpha\delta}{2}\left(g\owedge g+\sum_{\mu=1}^3\Phi_\mu^\mathcal{H}\owedge\Phi_\mu^\mathcal{H}\right)+\mathcal{R}_1.
\]
Note that in the degenerate case, the picture simplifies, since then we just have $\Rcpar=\Rc_1=\Rc^{g_N}$.
\end{rem}

We will now show some crucial properties of the spectrum of the introduced operators.
Before proving the next lemmas, we remark a few facts on the fundamental $2$-forms $\Phi_i$ of a $3$-$(\alpha,\delta)$-Sasaki structure.
Each $\Phi_i$ can be expressed as
\begin{equation}\label{Phi}
\Phi_i=-\frac{1}{2}\sum_{r=1}^{4n+3} e_r\wedge\varphi_i e_r,
\end{equation}
where $\{e_r, r=1,\ldots,4n+3\}$ is a local orthonormal frame. A straightforward computation shows that for every $i,j,k=1,2,3$,
\begin{equation}\label{phi_scalar_phi}
\langle\Phi_i,\Phi_j\rangle\,=(2n+1)\delta_{ij}
\end{equation}
and
\begin{align}\label{phi_scalar_xi}
\langle\Phi_i,\xi_{jk}\rangle=-\epsilon_{ijk},
\end{align}
where $\epsilon_{ijk}$ is the totally skew symmetric symbol. We will also use adapted bases in the following sense.
\begin{definition}
%-----------------------------------------------
An adapted basis of a $3$-$(\alpha,\delta)$-Sasaki manifold $M$ is a local orthonormal frame $\{e_1,\dots,e_{4n+3}\}$ of $M$ such that
\[
e_{i}=\xi_i, \quad e_{4l+i}=\varphi_ie_{4l} \quad \text{ for} \  i=1,2,3,\ l=1,\dots, n.
\]
\end{definition}
\noindent
Therefore, in an adapted basis, the horizontal part $\Phi_i^{\H}=\Phi_i+\xi_{jk}$ is expressed as
\begin{equation}
\label{2form-adapted}
\Phi_i^{\H}=-\frac{1}{4}\sum_{r=4}^{4n+3}e_r\wedge \varphi_i e_r+\varphi_j e_r\wedge \varphi_k e_r.
\end{equation}
\begin{lemma}\label{EVofRm}
%-----------------------------------------------
$\Rcperp $ has the only non-zero eigenvalue $2(n+2)$ with eigenspace generated by $\Phi_i-\xi_{jk}$ for $i=1,2, 3$.
\end{lemma}
\begin{proof}
By definition \eqref{Rcperp}, $\Phi_i-\xi_{jk}$ are the only eigenvectors with non-vanishing eigenvalue $|\Phi_i-\xi_{jk}|^2$. From \eqref{phi_scalar_phi} and \eqref{phi_scalar_xi} we obtain the eigenvalue $2(n+2)$.
\end{proof}
\begin{lemma}\label{propRpar}
%-----------------------------
The kernel of $\Rcpar$ contains the space generated by $Z\wedge\varphi_iZ+\varphi_jZ\wedge\varphi_kZ$, $Z\in \mathcal{H}$. In particular $\Phi_i^{\H}$, $\Phi_i$, $\Phi_i-\xi_{jk}$, $\Phi_i+(n+1)\xi_{jk}\in\ker\Rcpar$.
\end{lemma}
\begin{proof}
For $X,Y,Z\in \mathcal{H}$ compare \eqref{2Curv} with \eqref{Rcperp} to obtain
\begin{align*}
\langle\Rc (X\wedge Y),Z\wedge \varphi_iZ+\varphi_jZ\wedge \varphi_kZ\rangle&=-2\alpha\beta\Phi_i(X,Y)\|Z\|^2\\
&=\langle\alpha\beta\Rcperp (X\wedge Y),Z\wedge \varphi_iZ+\varphi_jZ\wedge \varphi_kZ\rangle
\end{align*}
for any even permutation $(ijk)$ of $(123)$. Thus, $Z\wedge\varphi_iZ+\varphi_jZ\wedge\varphi_kZ\in\ker \Rcpar$. The second part of the statement follows immediately from \eqref{2form-adapted} and $\Rcpar|_{(\Lambda^2\mathcal{H})^\perp}=0$.
\end{proof}

%\subsection{Semi-Definite Canonical Curvature}
%

%
As a first consequence we can obtain a distinguished set of eigenforms of the canonical curvature operator $\Rc$, which will have a special role in the characterization of the Einstein condition for a $3$-$(\alpha,\delta)$-Sasaki manifold (see \autoref{theo_einstein_case}).
\begin{theorem}\label{theo_eigenforms_R}
%----------------------------------------
 The curvature operator $\Rc$ of the canonical connection of any $3$-$(\alpha,\delta)$-Sasaki
manifold $(M,\varphi_i,\xi_i,\eta_i,g)$ admits the following six orthogonal eigenforms:
 \begin{itemize}
  \item $\Phi_i-\xi_{jk}=\Phi_i^{\H}-2\xi_{jk}$, $i=1,2,3$, eigenform with eigenvalue $2\alpha\beta(n+2)$,
  \item $\Phi_i+(n+1)\xi_{jk}=\Phi_i^{\H}+n\xi_{jk}$, $i=1,2,3$, eigenform with vanishing eigenvalue.
 \end{itemize}
\end{theorem}
\begin{proof}
%--------------------------------------
From \autoref{propRpar}, all the forms are in the kernel of  $\Rcpar$. Therefore, we only have to check that $\Phi_i-\xi_{jk}$ and $\Phi_i+(n+1)\xi_{jk}$ are eigenvectors of $\Rcperp$ with the respective eigenvalues. \autoref{EVofRm} provides just that under the observation $\langle \Phi_i+(n+1)\xi_{jk}, \Phi_i-\xi_{jk}\rangle=0$.
\end{proof}
For later use we observe that one can immediately obtain:
\begin{align}
\Rc(\Phi_i)&=\alpha\beta\Rcperp(\Phi_i)=2\alpha\beta (n+1)(\Phi_i-\xi_{jk})\label{RPhi},\\
\Rc(\xi_{jk})&=\alpha\beta\Rcperp(\xi_{jk})=-2\alpha\beta(\Phi_i-\xi_{jk})\label{Rxijk}.
\end{align}
\begin{prop}
The eigenvalues of the canonical curvature operator $\Rc$ satisfy
\begin{equation}\label{SpecR}
\mathrm{Spec}(\Rc)=\mathrm{Spec}(\Rc_\mathrm{par})\cup\{0,2\alpha\beta(n+2)\}
\end{equation}
whereas
\begin{equation}\label{SpecRN}
\mathrm{Spec}(\Rc^{g_N})\cup\{0\}=\mathrm{Spec}(\Rcpar)\cup\{0, 4\alpha\delta n\}.
\end{equation}
\end{prop}
\begin{proof}
\autoref{EVofRm} and \autoref{propRpar} show that $\Rc_\mathrm{par}$ and $\Rcperp$ share eigenspaces. Additionally, the eigenspaces to non-zero eigenvalues of $\Rcpar$ are inside the kernel of $\Rcperp$ and vice versa. Thus the eigenvalues of $\Rc$ are simply the union
\begin{equation*}
\mathrm{Spec}(\Rc)=\mathrm{Spec}(\Rc_\mathrm{par})\cup\{0,2\alpha\beta(n+2)\}
\end{equation*}
of the eigenvalues of $\Rc_\mathrm{par}$ and the eigenvalues $\{0,2\alpha\beta(n+2)\}$ of $\alpha\beta\Rcperp$.
For the second identity we set $\Rcperp^{\H}\coloneqq \sum_{i=1}^3\Phi_i^{\H}\otimes\Phi_i^{\H}$ in analogy to the first identity. Immediately we find that $\Rcperp^{\H}$ has the only non-zero eigenvalue $2n$ with eigenspace spanned by $\Phi_i^{\H}$, $i=1,2,3$. As before it follows that the eigenvalues of $\Rc^{g_N}$ are the union of those of $\Rc_\mathrm{par}$ and $2\alpha\delta\Rcperp^{\H}$. Thus we have \eqref{SpecRN}. Remark that $\mathrm{Spec}(\Rc)$ certainly contains $0$ as $\Phi_i^{\H}+n\xi_{jk}\in\ker \Rc$, by \autoref{theo_eigenforms_R}, while $\mathrm{Spec}(\Rc^{g_N})$ may or may not contain $0$.
\end{proof}
\begin{rem}\label{selfdual}
The discussion on the curvature operator $\Rc^{g_N}$ actually showed that the $\Phi_i$ are eigenvectors of $\Rc^{g_N}$ with eigenvalue $4\alpha\delta$. Thus, only now we proved that the canonical submersion of a $7$-dimensional $3$-$(\alpha,\delta)$-Sasaki manifold has a quaternionic K\"ahler base under stricter definition usually assumed. Compare the discussion ahead of \cite[Definition 12.2.12]{Boyer&Galicki}.
\end{rem}
The following theorem links the Riemannian curvature of the qK base to the canonical curvature of the total space, thus underlining the intricate relationship of these two connections.
\begin{theorem}\label{semidefinite}
%----------------------------------------------
Let $M$ be a $3$-$(\alpha,\delta)$-Sasaki manifold with canonical submersion $\pi\colon M\to N$.
\begin{enumerate}[a)]
\item If $N$ has a non-negative Riemannian curvature operator $\Rc^{g_N}\geq 0$, then $\Rc$ is non-negative if and only if $\alpha\beta\geq 0$.
\item If $N$ has a non-positive Riemannian curvature operator $\Rc^{g_N}\leq 0$, then $\Rc$ is non-positive.
\end{enumerate}
\end{theorem}
\begin{proof}
By \eqref{SpecRN} if $\Rc^{g_N}$ is either non-negative or non-positive then so is $\Rcpar$ and the sign of $\alpha\delta$. Using \eqref{SpecR} we obtain part a) directly. For part b) note that if $\alpha\delta\leq 0$ then $\alpha\beta=2\alpha\delta-4\alpha^2<0$.
\end{proof}
%In the next two sections we will determine some eigenvalues and corresponding eigenforms of $\Rc$ and $\Rc^g$.
%The problem of strong positivity of the Riemannian curvature will be investigated for homogeneous
%$3$-$(\alpha,\delta)$-Sasaki manifolds in \autoref{section definiteness}. 
%!TeX root=ADS3adCurv-bb.tex
%-----------------------------------------------------------------------------------------
\section{Curvature Eigenforms and the Einstein Condition}\label{eigenvalues}
In \cite{AgrDil} the authors computed the Ricci tensor \eqref{riccig_complete} implying that a $3$-$(\alpha,\delta)$-Sasaki manifold is Riemannian Einstein if and only if either $\delta=\alpha$ or $\delta=\alpha(2n+3)$.
The aim of this section is to provide an additional characterization of the Einstein condition for $3$-$(\alpha,\delta)$-Sasaki manifolds through special eigenforms of the curvature operator $\Rc^g$. More precisely, we shall prove:
%Recall that the Riemannian Ricci tensor of a $3$-$(\alpha,\delta)$-Sasaki is given by \eqref{riccig_complete}, %implying that the manifold is Riemannian Einstein if and only if $\delta=\alpha$ or $\delta=\alpha(2n+3)$.
%In the following we provide one further characterization.
%
\begin{theorem}\label{theo_einstein_case}
%-----------------------------------------------
Let $(M,\varphi_i,\xi_i,\eta_i,g)$ be a $3$-$(\alpha,\delta)$-Sasaki manifold. Then, the following conditions are equivalent:
\begin{enumerate}
\item[a)] each $2$-form $\Phi_i-\xi_{jk}$ is an eigenform of $\Rc^g$;
\item[b)] each $2$-form $\Phi_i+(n+1)\xi_{jk}$ is an eigenform of $\Rc^g$;
\item[c)] either $\delta=\alpha$ or $\delta=(2n+3)\alpha$;
\item[d)] $(M,g)$ is Einstein.
\end{enumerate}
If $\delta=\alpha$, that is in the $3$-$\alpha$-Sasaki case, the six orthogonal eigenforms
$\Phi_i-\xi_{jk}$, $\Phi_i+(n+1)\xi_{jk}$ admit the same eigenvalue $\lambda=\alpha^2$. If $\delta=\alpha(2n+3)$,
then each $\Phi_i-\xi_{jk}$ has eigenvalue $\lambda_1=\alpha^2(8n^2+16n+9)$, while each $\Phi_i+(n+1)\xi_{jk}$ has
eigenvalue $\lambda_2=\alpha^2(2n+1)^2$.
\end{theorem}
\begin{rem}
Together with \autoref{theo_eigenforms_R} we observe that exclusively in the Einstein case $\Phi_i-\xi_{jk}$ and $\Phi_i+(n+1)\xi_{jk}$ are joint eigenforms of $\Rc$, $\Rc^g$ and $\G_T+\Sc_T$. Since $\Phi_i+(n+1)\xi_{jk}\in\ker\Rc$ we have that the corresponding eigenvalue of $\G_T+\Sc_T$ is $4\lambda$, or $4\lambda_2$ respectively.
\end{rem}
In order to prove \autoref{theo_einstein_case} we will determine throughout the next propositions how $\mathcal{R}^g$ acts on the forms $\Phi_i$ and $\xi_{jk}$. Recall that by \eqref{Rg-Rshort} the curvature operators $\mathcal{R}^g$ and $\mathcal{R}$ are related by the operators $\Sc_T$ and ${\mathcal G}_T$ defined in \autoref{Def-S_TG_T}. They act on the forms $\Phi_i$ and $\xi_{jk}$ as follows.
\begin{prop}\label{propositionST}
%---------------------------------
Let $(M,\varphi_i,\xi_i,\eta_i,g)$ be a $3$-$(\alpha,\delta)$-Sasaki manifold.
The torsion $T$ of the canonical connection satisfies the following:
%\begin{align}
%dT(\xi_{jk})&=4\alpha\beta(\Phi_i+\xi_{jk}),\label{dT_xi}\\
%dT(\Phi_i)&=\left\{8\alpha^2(2n+1)-4\alpha\beta\right\}\Phi_i+\left\{8\alpha^2(2n+1)+4\alpha\beta(2n-1)\right\}\xi_{jk}.\label{dT_Phi}
%\end{align}
\begin{align}
\Sc_T(\xi_{jk})&=2\alpha\beta(\Phi_i+\xi_{jk}),\label{dT_xi}\\
\Sc_T(\Phi_i)&=\left\{4\alpha^2(2n+1)-2\alpha\beta\right\}\Phi_i+\left\{4\alpha^2(2n+1)+2\alpha\beta(2n-1)\right\}\xi_{jk}.\label{dT_Phi}
\end{align}
\end{prop}
\begin{proof}
%-----------------
First we show that for every vector fields $X,Y$ and for every even permutation $(ijk)$ of $(123)$
\begin{equation}\label{dT_aux}
dT(X,Y,\xi_i,\xi_j)=4\alpha\beta\{\Phi_k(X,Y)+(\eta_i\wedge\eta_j)(X,Y)\},
\end{equation}
which is equivalent to \eqref{dT_xi}, taking into account \eqref{S_T}. Indeed, we compute
\begin{align*}
&\sum_{l=1}^3(\Phi_l\wedge\Phi_l)(X,Y,\xi_i,\xi_j)\\
%&=2\sum_{l}\{\Phi_l(X,Y)\Phi_l(\xi_i,\xi_j)+\Phi_l(X,\xi_i)\Phi_l(\xi_j,Y)+\Phi_l(X,\xi_j)\Phi_l(Y,\xi_i)\}\\
&=2\{\Phi_k(X,Y)\Phi_k(\xi_i,\xi_j)+\Phi_k(X,\xi_i)\Phi_k(\xi_j,Y)+\Phi_k(X,\xi_j)\Phi_k(Y,\xi_i)\}\\
&=2\{-\Phi_k(X,Y)+\eta_j(X)\eta_i(Y)-\eta_i(X)\eta_j(Y)\}\\
&={}-2\{\Phi_k(X,Y)+(\eta_i\wedge\eta_j)(X,Y)\}.
\end{align*}
We can also compute
\begin{align*}
&\cyclic{l,m,n}(\Phi_l\wedge\eta_m\wedge\eta_n)(X,Y,\xi_i,\xi_j)\\
&=\cyclic{l,m,n}\{\Phi_l(X,Y)\eta_{mn}(\xi_i,\xi_j)+\Phi_l(X,\xi_i)\eta_{mn}(\xi_j,Y)+\Phi_l(X,\xi_j)\eta_{mn}(Y,\xi_i)\\
&\qquad+\Phi_l(\xi_i,\xi_j)\eta_{mn}(X,Y)+\Phi_l(\xi_j,Y)\eta_{mn}(X,\xi_i)+\Phi_l(Y,\xi_i)\eta_{mn}(X,\xi_j)\}\\
&=\Phi_k(X,Y)-\Phi_k(X,\xi_i)\eta_i(Y)-\Phi_k(X,\xi_j)\eta_j(Y)\\
&\qquad{}-\eta_{ij}(X,Y)-\Phi_k(\xi_j,Y)\eta_j(X)+\Phi_k(Y,\xi_i)\eta_i(X)\\
&=\Phi_k(X,Y)-\eta_j(X)\eta_i(Y)+\eta_i(X)\eta_j(Y)-\eta_{ij}(X,Y)-\eta_i(Y)\eta_j(X)+\eta_j(Y)\eta_i(X)\\
&=\Phi_k(X,Y)+\eta_{ij}(X,Y).
\end{align*}
Therefore, using \eqref{dT} we get \eqref{dT_aux}.

In order to prove \eqref{dT_Phi}, first we show that for every $X,Y\in\Gamma(\H)$
\begin{equation}\label{ricci3}
\sum_{r=1}^{4n+3}dT(X,Y,e_r,\varphi_ie_r)=\{-16\alpha^2(2n+1)+8\alpha\beta\}\Phi_i(X,Y),
\end{equation}
where $e_r$, $r=1,\ldots,4n+3$, is a local orthonormal frame. Indeed, we can compute
\begin{align*}
&\sum_{r=1}^{4n+3}\sum_{l=1}^3(\Phi_l\wedge\Phi_l)(X,Y,e_r,\varphi_ie_r)\\
&=2\sum_{r}\sum_{l}\{\Phi_l(X,Y)\Phi_l(e_r,\varphi_ie_r)+\Phi_l(X,e_r)\Phi_l(\varphi_ie_r,Y)+\Phi_l(X,\varphi_ie_r)\Phi_l(Y,e_r)\}\\
&=2\sum_r\Phi_i(X,Y)\Phi_i(e_r,\varphi_ie_r)+2\sum_lg(\varphi_lX,\varphi_i\varphi_lY)-2\sum_lg(\varphi_i\varphi_lX,\varphi_lY)\\
&=-2(4n+2)\Phi_i(X,Y)-2g(\varphi_iX,Y)-4g(X,\varphi_iY)+2g(X,\varphi_iY)+4g(\varphi_iX,Y)\\
&=-8(n+1)\Phi_i(X,Y).
\end{align*}
We also compute
\begin{align*}
&\sum_{r=1}^{4n+3}\cyclic{l,m,n}(\Phi_l\wedge\eta_m\wedge\eta_n)(X,Y,e_r,\varphi_ie_r)\\
&=\sum_{r}\cyclic{l,m,n}\Phi_l(X,Y)(\eta_m\wedge\eta_n)(e_r,\varphi_ie_r)\\
&=\Phi_i(X,Y)\{(\eta_j\wedge\eta_k)(\xi_j,\xi_k)+(\eta_j\wedge\eta_k)(\xi_k, -\xi_j)\}\\
%&\quad+\Phi_j(X,Y)\{(\eta_k\wedge\eta_i)(\xi_j,\xi_k)+(\eta_k\wedge\eta_i)(\xi_k, -\xi_j)\}\\
%&\quad+\Phi_k(X,Y)\{(\eta_i\wedge\eta_j)(\xi_j,\xi_k)+(\eta_i\wedge\eta_j)(\xi_k, -\xi_j)\}\\
&=2\Phi_i(X,Y).
\end{align*}
Applying \eqref{dT},we can deduce that
\[
\sum_{r=1}^{4n+3}dT(X,Y,e_r,\varphi_ie_r)=-32\alpha^2(n+1)\Phi_i(X,Y)-16\alpha(\alpha-\delta)\Phi_i(X,Y),
\]
which gives \eqref{ricci3}, being $\beta=2(\delta-2\alpha)$. Now, from \eqref{S_T}, \eqref{Phi} and \eqref{ricci3}, we have
\begin{align*}
\langle\Sc_T(\Phi_i),X\wedge Y\rangle&=-\frac12\sum_{r=1}^{4n+3}\langle\Sc_T(e_r\wedge\varphi_ie_r),X\wedge Y\rangle=-\frac14\sum_{r=1}^{4n+3}dT(X,Y,e_r,\varphi_ie_r)\\
&=\{4\alpha^2(2n+1)-2\alpha\beta\}\langle\Phi_i,X\wedge Y\rangle.
\end{align*}
%
%\begin{align*}
% \langledT(\Phi_i),X\wedge Y\rangle&=-\frac12\sum_{r=1}^{4n+3}\langledT(e_r\wedge\varphi_ie_r),X\wedge Y\rangle
% =-\frac12\sum_{r=1}^{4n+3}dT(X,Y,e_r,\varphi_ie_r)\\
%&=\{8\alpha^2(2n+1)-4\alpha\beta\}\langle\Phi_i,X\wedge Y\rangle.
%\end{align*}
%
From \eqref{dT}, one can see that for every $X,Y,Z\in\Gamma(\H)$ and $r,s,t=1,2,3$
\[dT(X,Y,Z,\xi_r)=0, \qquad dT(X,\xi_r,\xi_s,\xi_t)=0,\]
which imply that
\[
\langle\Sc_T(\Phi_i),X\wedge\xi_s\rangle=-\frac12 \sum_{r=1}^{4n+3}\langle\Sc_T(e_r\wedge\varphi_ie_r),X\wedge\xi_s\rangle
=-\frac14\sum_{r=1}^{4n+3}dT(e_r,\varphi_ie_r,X,\xi_s)=0.
\]
Finally, using \eqref{dT_xi} and \eqref{phi_scalar_phi},
\begin{align*}
\langle\Sc_T(\Phi_i),\xi_j\wedge\xi_k\rangle&=\langle\Sc_T(\xi_j\wedge\xi_k),\Phi_i\rangle=2\alpha\beta\langle\Phi_i+\xi_{jk},\Phi_i\rangle\\
&=2\alpha\beta(2n+1+\Phi_i(\xi_j,\xi_k))=4\alpha\beta n
\end{align*}
and analogously, $\langle\Sc_T(\Phi_i),\xi_i\wedge\xi_j\rangle=\langle\Sc_T(\Phi_i),\xi_i\wedge\xi_k\rangle=0$,
thus completing the proof of \eqref{dT_Phi}.
\end{proof}
\begin{prop}\label{propG_T}
%-------------------------------
Let $(M,\varphi_i,\xi_i,\eta_i,g)$ be a $3$-$(\alpha,\delta)$-Sasaki manifold.
The torsion $T$ of the canonical connection satisfies the following:
\begin{align}
\G_T(\xi_{jk})&=(\beta-4\alpha)\left\{2\alpha\Phi_i+(\beta-2\alpha)\xi_{jk}\right\},\label{S_Txi}\\
\G_T(\Phi_i)&=\left\{8\alpha^2(n+1)-2\alpha\beta\right\}\Phi_i+\left\{-8\alpha^2(n+1)
-\beta^2+2\alpha\beta(2n+3)\right\}\xi_{jk}.\label{S_TPhi}
\end{align}
\end{prop}
\begin{proof}
%--------------------
By direct computation using \eqref{torsion01} or \eqref{torsion02}, one easily gets that for all
vector fields $X,Y$ and for every even permutation $(i,j,k)$ of $(1,2,3)$
\[
T(X,Y,\xi_i)=2\alpha\Phi_i(X,Y)+2(\delta-3\alpha)(\eta_j\wedge\eta_k)(X,Y),\]
\[T(\xi_j,\xi_k)=2(\delta-4\alpha)\xi_i.
\]
Hence,
\begin{align*}
\langle\G_T(\xi_{jk}),X\wedge Y\rangle&=g(T(\xi_j,\xi_k),T(X,Y))=2(\delta-4\alpha)T(X,Y,\xi_i)\\
&=2(\delta-4\alpha)\{2\alpha\Phi_i(X,Y)+2(\delta-3\alpha)(\eta_j\wedge\eta_k)(X,Y)\},
\end{align*}
which gives \eqref{S_Txi}, since $\beta=2(\delta-2\alpha)$.

In order to prove \eqref{S_TPhi}, notice that for every $X,Y,Z,V\in\Gamma(\H)$, applying \eqref{torsion02}, we have
\[
G_T(X,Y,Z,V)=4\alpha^2\sum_{l=1}^3\Phi_l(X,Y)\Phi_l(Z,V).
\]
Then, choosing a local orthonormal frame of type $e_r$, $r=1,\ldots,4n$, $\xi_i,\xi_j,\xi_k$, and using
\eqref{Phi}, for every $X,Y\in\Gamma(\H)$ we have
\begin{align*}
\langle\G_T(\Phi_i),X\wedge Y\rangle&=-\frac12\sum_{r=1}^{4n}G_T(e_r,\varphi_ie_r,X,Y)-G_T(\xi_j,\xi_k,X,Y)\\
&=-2\alpha^2\sum_{r=1}^{4n}\sum_{l=1}^3\Phi_l(e_r,\varphi_ie_r)\Phi_l(X,Y)-2\alpha(\beta-4\alpha)\Phi_i(X,Y)\\
&=8\alpha^2n\Phi_i(X,Y)-2\alpha(\beta-4\alpha)\Phi_i(X,Y)\\
&=\{8\alpha^2(n+1)-2\alpha\beta\}\langle\Phi_i,X\wedge Y\rangle.
\end{align*}
Now, from \eqref{torsion02}, we also have that for every $X,Y,Z\in\Gamma(\H)$ and $r,s,t=1,2,3$,
\[
G_T(X,Y,Z,\xi_r)=0,\qquad G_T(\xi_r,\xi_s,\xi_t,X)=0,
\]
which give
\[
\langle\G_T(\Phi_i),X\wedge \xi_s\rangle=-\frac12\sum_{r=1}^{4n}G_T(e_r,\varphi_ie_r,X,\xi_s)-G_T(\xi_j,\xi_k,X,\xi_s)=0.
\]
Finally, using \eqref{S_Txi} and \eqref{phi_scalar_phi},
\[
\langle\G_T(\Phi_i),\xi_{jk}\rangle=\langle\G_T(\xi_{jk}),\Phi_i\rangle=(\beta-4\alpha)\{2\alpha(2n+1)-\beta+2\alpha\},
\]
which is coherent with \eqref{S_TPhi}. Analogously,
\[
\langle\G_T(\Phi_i),\xi_{ij}\rangle=\langle\G_T(\Phi_i),\xi_{ki}\rangle=0,
\]
thus completing the proof of \eqref{S_TPhi}.
\end{proof}

\begin{prop}
%----------------
Let $(M,\varphi_i,\xi_i,\eta_i,g)$ be a $3$-$(\alpha,\delta)$-Sasaki manifold.
Then, the Riemannian curvature of $M$ satisfies
\begin{align}
\Rc^g(\xi_{jk})&=-\left(2\alpha^2+\alpha\beta\right)\Phi_i+\big(2\alpha^2+\alpha\beta+\frac14\beta^2\big)\xi_{jk},\label{Rgxi}\\
\Rc^g(\Phi_i)&=\left\{\alpha^2(4n+3)+\alpha\beta(2n+1)\right\}\Phi_i-\big(\alpha+\frac12\beta\big)^2\xi_{jk}.\label{RgPhi}
\end{align}
\end{prop}
\begin{proof}
%------------------
Since $\Rc^g$ and $\Rc$ are related by \eqref{Rg-Rshort}, the result follows from direct computations
using equations \eqref{Rxijk}, \eqref{S_Txi}, \eqref{dT_xi}, and equations \eqref{RPhi}, \eqref{S_TPhi}, \eqref{dT_Phi}.
\end{proof}
\begin{rem}
%-----------------
Being $\beta=2(\delta-2\alpha)$, equations \eqref{Rgxi} and \eqref{RgPhi} can be rephrased as
\begin{align}
\Rc^g(\xi_{jk})&=2\alpha(\alpha-\delta)\Phi_i+\{\alpha^2+(\alpha-\delta)^2\}\xi_{jk},\label{Rgxi1}\\
\Rc^g(\Phi_i)&=\{\alpha(\delta-\alpha)(4n+1)+\alpha\delta\}\Phi_i-(\alpha-\delta)^2\xi_{jk}.\label{RgPhi1}
\end{align}
\end{rem}
We have now gathered all necessary results to give a proof of the main theorem.
\begin{proof}[Proof of \autoref{theo_einstein_case}]
%---------------
The equivalence of c) and d) is known, see \cite[Proposition 2.3.3]{AgrDil}. From \eqref{Rgxi1} and \eqref{RgPhi1}, we have that
$\Rc^g(\Phi_i-\xi_{jk})=a\Phi_i+b\xi_{jk}$ with
\[
a=\alpha(\delta-\alpha)(4n+1)+\alpha\delta-2\alpha(\alpha-\delta),\qquad b=-\alpha^2-2(\alpha-\delta)^2,
\]
which give \[a+b=2(\delta-\alpha)\{\alpha(2n+3)-\delta\}.\] Then  $\Phi_i-\xi_{jk}$ is an eigenform of
$\Rc^g$ if and only if $a+b=0$, that is $\delta=\alpha$ or $\delta=(2n+3)\alpha$. In particular, if $\delta=\alpha$,
the corresponding eigenvalue is $\lambda=a=\alpha^2$. If $\delta=\alpha(2n+3)$, the eigenvalue is
$\lambda_1=a=\alpha^2(8n^2+16n+9)$.

Analogously, we can compute $\Rc^g(\Phi_i+(n+1)\xi_{jk})=a'\Phi_i+b'\xi_{jk}$
with \[a'=\alpha(\delta-\alpha)(4n+1)+\alpha\delta+2\alpha(\alpha-\delta)(n+1),\quad
b'=-(\alpha-\delta)^2+(n+1)\alpha^2+(\alpha-\delta)^2(n+1),\]
from which
\[
(n+1)a'-b'=n(\delta-\alpha)\{\alpha(2n+3)-\delta\},
\]
thus proving the equivalence of b) and c). If $\delta=\alpha$, then the eigenform $\Phi_i+(n+1)\xi_{jk}$
has eigenvalue $\lambda=a'=\alpha^2$. If $\delta=\alpha(2n+3)$, the corresponding eigenvalue is $\lambda_2=a'=\alpha^2(2n+1)^2$.
\end{proof}

%\[\langle\Rc^g(X\wedge Y),Z\wedge V\rangle=-R^g(X,Y,Z,V)=-g(R^g(X,Y)Z,V),\]
%\[\langle\Rc(X\wedge Y),Z\wedge V\rangle=-R(X,Y,Z,V)=-g(R(X,Y)Z,V)\]
%for every vector fields $X,Y,Z,V$. 
%!TeX root=ADS3adCurv-bb.tex
 %-----------------------------------------------------------------------------------------------------------------------
\section{Definiteness of Curvature Operators}\label{section definiteness}
%-----------------------------------------------------------------------------------------------------------------------
%-----------------------------------------------------------------------------------------------------------------------
\subsection{Strongly Positive Curvature}
%-----------------------------------------------------------------------------------------------------------------------
%
%
We now investigate strongly non-negative and even strongly positive curvature on $(M,g)$. Recall that, by \eqref{Rg-Rshort}, the curvature operators $\Rc$ and $\Rc ^g$ are related by
\begin{equation*}
\Rc^g=\Rc+\frac14 \G_T+\frac14 \Sc_T.
\end{equation*}
In particular, $(M,g)$ is strongly non-negative with 4-form $-\frac 14 \sigma_T$ if and only if
\[
\Rc+\frac 14 \G_T\geq 0.
\]
Observe that $\G_T$ is non-negative by definition, so we have directly strong non-negativity if $\Rc$ is non-negative. \autoref{semidefinite} thus yields (recall that $\beta := 2 (\delta -2\alpha)$)
\begin{cor}\label{nonnegop}
Let $M$ be a $3$-$(\alpha,\delta)$-Sasaki manifold with $\alpha\beta\geq 0$ and $\Rc^{g_N}\geq 0$. Then $(M,g)$ is strongly non-negative with $4$-form $-\frac 14\sigma_T$.
\end{cor}

As we later see this will be sufficient for homogeneous spaces, but in general the condition $\Rc^{g_N}\geq 0$ is too strong. However, we can relax the condition on the base to strong non-negativity, but we need an additional assumption on the $4$-form. To do all this,  we need some notation.

%Given an adapted basis we consider the quaternionic spaces $\H_l=\mathrm{span}\{e_{4l},e_{4l+1},e_{4l+2},e_{4l+3}\}$.
%
%We further obtain from an adapted basis in a canonical way a local frame of $\Lambda^2\H_i$, $\H_i\wedge\H_j$ and $\H_i\wedge\V$.
%
For $i=1,2,3$ denote the $2$-dimensional spaces $N_i:=\mathrm{span}\{\Phi_i^{\H},\xi_{jk}\}$. Then decompose the space of $2$-forms into orthogonal subbundles
\[
\Lambda^2M=\Lambda^2_1\oplus \Lambda^2_2\oplus \Lambda^2_3,
\]
where the $\Lambda_i^2$ are given by
\begin{gather*}
\Lambda^2_1=\bigoplus_{i=1}^3N_i,\quad
\Lambda^2_2=\Lambda^2\H\cap\{\Phi_1^{\H},\Phi_2^{\H},\Phi_3^{\H}\}^\perp,\quad
\Lambda^2_3=\V\wedge \H.
\end{gather*}
For a linear map $A\colon \Lambda^2M\to\Lambda^2M$ we denote $A_1\coloneqq A|_{\Lambda^2_1}$ and correspondingly for the other spaces.

Let us motivate this decomposition. The obvious $\Lambda^2(\V\oplus\H)=\Lambda^2\V\oplus\V\wedge\H\oplus\Lambda^2\H$ motivates $\Lambda^2_3=\V\wedge\H$, in particular since $\Rc|_{\Lambda^2_3}=0$. However, $\Rc$ does not restrict to $\Lambda^2\V$ and $\Lambda^2\H$, but to $\Lambda^2_1$ and $\Lambda^2_2$. In fact, the characterization $\Rc=\alpha\beta\Rcperp+\Rcpar$ is with respect to $\Lambda^2_1$ and $\Lambda_2^2$ as noted in the proof of \autoref{theo_eigenforms_R}. The space $\Lambda_1^2$ can be seen as controlled by the $3$-$(\alpha,\delta)$-Sasaki structure, while $\Lambda_2^2$ resonates the geometry of the base $N$. This is emphasized by the fact that the common eigenforms discussed in \autoref{eigenvalues} all lie in $\Lambda^2_1$.
\begin{df}
We call a $4$-form $\omega\in\Lambda^4N$ on a quaternionic K\"ahler space $N$ \emph{adapted with minimal  eigen\-value $\nu\in \R$} if for every point $p\in N$ the quaternionic bundle $\mathcal{Q}$ lies in the $\nu_p$-eigenspace of $\omega_p$, considered as an operator $\Lambda^2T_pN\to \Lambda^2T_pN$, where the eigenvalues $\nu_p$ are bounded below by $\nu$.
\end{df}
Given a $3$-$(\alpha,\delta)$-Sasaki manifold with canonical submersion $\pi\colon M\to N$ the adaptedness of $\omega\in \Lambda^4 N$ implies that $\pi^*\omega$ admits a block diagonal structure $\pi^*\omega=(\pi^*\omega)_1\oplus(\pi^*\omega)_2$. Note that $(\pi^*\omega)_3$ vanishes trivially as $\pi_*\V=0$.
%
%This appears to be a feasible condition to lift the strong non-negativity of the qK base to the total space. Name is my first impression but I'm happy for better ideas.
%
We can now state the result with strong non-negativity on the base. It turns out the proof for strong positivity is almost identical. Hence we formulate the result in the latter case.
\begin{theorem}\label{stronglypos}
%-----------------------------------------------
Let $M$ be a positive $3$-$(\alpha,\delta)$-Sasaki manifold. Assume that the base $N$ of the canonical submersion is strongly positive with respect to an adapted $4$-form $\omega$ with minimal eigenvalue $\nu$. Suppose further that the conditions
 \begin{align}\label{cisbig}
\delta^2+4n\alpha\delta-6n\alpha^2+\nu>0,\qquad 4n\alpha(\delta-2\alpha)^3+\delta^2\nu>0,\qquad \text{and}\qquad\delta>2\alpha
\end{align}
are satisfied.
Then $M$ is strongly positive with $4$-form  $\pi^*\omega-(\frac 14+\varepsilon) \sigma_T$ for some $\varepsilon>0$ sufficiently small.
\end{theorem}
\begin{cor}
If the base $N$ in the situation of \autoref{stronglypos} is only strongly non-negative and satisfies the non strict inequalities \eqref{cisbig} for $r$, then $M$ is strongly non-negative with $4$-form $\pi^*\omega-\frac 14\sigma_T$.
\end{cor}
The corollary will be proved as a byproduct of \autoref{stronglypos}.
\begin{rem}\label{nonscalingHdeform}
Observe that the conditions \eqref{cisbig} will be fulfilled for $\delta/\alpha\gg2$ sufficiently big. This can be achieved by $\H$-homothetic deformation, compare \cite[Section $2.3$]{AgrDil}. In fact, the horizontal structure is only changed by global scaling via a parameter $a$ inversely proportional to $\alpha\delta$. However, fixing $a$ we can scale the Reeb orbits by a parameter $c$ implying a quadratic change in $\frac{\delta}{\alpha}$. Therefore, such a $\H$-homothetic deformation does not change the horizontal structure and thereby fixes $\nu$, but it increases the leading term of both polynomial conditions.
\end{rem}
To prove these results we need a more deliberate investigation of how $\G_T$ acts on the spaces $\Lambda_i^2$.
From equation \eqref{torsion02} it follows that the torsion $T$ of the canonical connection satisfies
\[
 X\wedge Y\in\V\wedge\H\ \Rightarrow\  T(X, Y)\in\H \quad \text{and}\quad X\wedge Y\in\Lambda^2\V\oplus\Lambda^2\H\ \Rightarrow\ T(X,Y)\in\V.
\]
Thus, $\G_T$ preserves $\Lambda^2_3=\V\wedge\H$ and by \autoref{propG_T} $\Lambda^2_1$ as well. Therefore $\G_T$ splits into a direct sum of operators $\G_1\oplus\G_2\oplus \G_3$ on $\Lambda^2_1\oplus \Lambda^2_2\oplus\Lambda^2_3$.

Consider some adapted basis $e_r$, $r=1,\dots, 4n+3$ of $M$. We may define the quaternionic spaces $\H_l=\mathrm{span}\{e_{4l},e_{4l+1},e_{4l+2},e_{4l+3}\}$, $l=1,\dots, n$, and accordingly we have
\begin{gather*}
\Lambda_2^2=\left(\bigoplus_{l=1}^n\Lambda^2\H_l\right)\!\cap\!\langle\Phi_i^\H\rangle^\perp\oplus\bigoplus_{k<l}\H_k\wedge\H_l,
\qquad \Lambda^2_3=\bigoplus_{l=1}^n\V\wedge\H_l.
\end{gather*}
Note that these descriptions depend on the choice of adapted basis unlike the spaces $\Lambda^2_i$ themselves.
\begin{lemma}\label{S4}
%-----------------------------------------------
The linear operator $\G_2$ vanishes. The operator $\G_3$ has the unique non-vanishing eigenvalue $12\alpha^2$ with eigenspace generated by $e_l\wedge\xi_1+\varphi_3e_l\wedge\xi_2-\varphi_2e_l\wedge\xi_3$, $l=4,\dots,4n+3$.
\end{lemma}
\begin{proof}
%-----------------------------------------------
The space $\left(\bigoplus\Lambda^2\H_l\right)\cap\langle\Phi_i^\H\rangle^\perp\subset\Lambda^2_2$ is spanned by $e_l\wedge\varphi_ie_l-e_r\wedge\varphi_i e_r$, $r,l=4,\dots, 4n+3$, $i=1,2, 3$. $\G_2$ vanishes on these. Indeed, by expression \eqref{torsion02} of $T$
\begin{equation}\label{adaptorsion}
T(e_l\wedge\varphi_i e_l)=2\alpha\Phi_i(e_l,\varphi_i e_l)\xi_i=-2\alpha\xi_i=T(e_r\wedge \varphi_i e_r)
\end{equation}
for any $l=4,\dots, 4n+3$. $\G_2$ vanishes on $\H_k\wedge\H_l$ as well since $\Phi_i(X,Y)=0$ whenever $X$ and $Y$ are in different quaternionic subspaces.

We observe that $\G_3$ is the sum of $n$ identical copies $\hat{\G}_3$ for each space $\V\wedge\H_i$, since  $T(X,Y)\in \H_i$ if $X\in\H_i$ and $Y\in\V$.
Again using \eqref{torsion02}, we find
\begin{align}\label{torsionVH}
T(e_l\wedge\xi_i)=2\alpha\varphi_ie_l=T(\varphi_k e_l\wedge\xi_j)=-T(\varphi_j e_l\wedge\xi_k).
\end{align}
In particular,
\[
\G_T(e_r\wedge\xi_i)=\G_T(\varphi_k e_r\wedge\xi_j)=\G_T(-\varphi_j e_r\wedge\xi_k)=4\alpha^2(e_r\wedge\xi_1+\varphi_3 e_r\wedge\xi_2-\varphi_2 e_r\wedge\xi_3)
\]
 for the adapted basis $e_{4l},\dots,e_{4l+3}$ of $\H_l$. Hence, the vectors $e_r\wedge\xi_1+\varphi_3 e_r\wedge\xi_2-\varphi_2 e_r\wedge\xi_3$ are $4$ linearly independent eigenvectors with eigenvalue $12\alpha^2$. In fact, these are all the eigenvectors with non-zero eigenvalues, since \eqref{torsionVH} shows that
\[
4\leq\mathrm{rk}(\hat{\G}_3)=\dim T(\V\wedge \H_l)\leq\frac {\dim\V\wedge \H_l} {3}=\frac{12}{3}=4.\qedhere
\]
\end{proof}
This implies that analogous to $\Rcperp$ the sum $\alpha\beta\Rcperp+\frac 14 \G_ T$ is orthogonal to $\Rc_\mathrm{par}$, i.e.\ it is trivial on the space $\Lambda^2_2$ where $\Rcpar$ is non-trivial. It is now time to include the $4$-form $\omega$.%In particular, $\Rc+\frac 14 \G_T$ can never be non-negative unless $\Rc_\mathrm{par}$ is. %This, in fact, rules out strong non-negativity for negative homogeneous $3$-$(\alpha,\delta)$-Sasakian manifolds with respect to the $4$-form $-\frac 14\sigma_T$.

\begin{lemma}\label{sumonL1}
%-----------------------------------------------
The operator $\alpha\beta\Rcperp+\frac 14 \G_1+(\pi^*\omega)_1$ is positive definite on $N_i$, $i=1,2,3$, if $\alpha$ and $\delta$ satisfy the polynomial conditions in \eqref{cisbig}, i.e.~$\delta^2+4n\alpha\delta-6n\alpha^2+\nu>0$ and $4n\alpha(\delta-2\alpha)^3+\delta^2\nu>0$.
\end{lemma}
\begin{proof}
%-----------------------------------------------
From \autoref{propG_T} we obtain
%From \eqref{torsion01} and \eqref{adaptorsion} we compute that
\begin{align*}
\G_T(\Phi_i^\H)=\G_T(\Phi_i+\xi_{jk})&=8\alpha^2n\Phi_i^\H+8\alpha n(\delta-4\alpha)\xi_{jk}\\
\G_T(\xi_{jk})&=4\alpha(\delta-4\alpha)\Phi_i^{\H}+4(\delta-4\alpha)^2\xi_{jk}.
\end{align*}
By adaptedness of $\omega$ at every point the two-forms $\Phi_i^\mathcal{H}\in\pi^*\mathcal{Q}$ lie inside some eigenspace with eigenvalue $\nu_p\geq\nu\in \R$.
Thus on $N_i$ with respect to the orthonormal basis $\frac{1}{\sqrt{2n}}\Phi_i^\H$ and $\xi_{jk}$ the sum takes the matrix form
\begin{align*}
\alpha\beta\Rcperp+\frac 14 \G_T+(\pi^*\omega)_1=
\begin{pmatrix}
2n\alpha(2\delta-3\alpha)+\nu_p & -\sqrt{2n}\alpha(3\delta-4\alpha)\\
-\sqrt{2n}\alpha(3\delta-4\alpha) & \delta^2
\end{pmatrix}.
\end{align*}
Now the restriction to the 2-dimensional space $N_i$ is positive (non-negative) if and only if both the determinant and the trace are. We have
\begin{align*}
\mathrm{tr}_{N_i}\left(\alpha\beta\Rcperp+\frac 14 \G_T+(\pi^*\omega)_1\right)=2n\alpha(2\delta-3\alpha)+\delta^2+\nu_p.
\end{align*}
Since the $\nu_p$ are bounded below by $\nu$ the trace is positive if the quadratic polynomial in $\delta$ satisfies $\delta^2+4n\alpha\delta-6n\alpha^2+\nu>0$.
%Setting $r=\frac{\delta}{\alpha}\in \R$ one obtains the quadratic polynomial
%\begin{align*}
%\mathrm{tr}_{N_i}\left(\alpha\beta\Rcperp+\frac 14 \G_T+(\pi^*\omega)_1\right)=\alpha^2(r^2+4nr-6n)+\nu_p.
%\end{align*}
%Clearly, positivity holds if $r$ satisfies $r^2+4nr-6n+\nu>0$.
%
%Linke Seite ist monoton steigend und konvergiert gegen 3/2...
%
The determinant is given by
\begin{align*}
\det\Bigg(\left(\alpha\beta\Rcperp+\frac 14\G_T+(\pi^*\omega)_1\right)\Big|_{N_i}\Bigg)
&=4n\alpha\delta^3-6n\alpha^2\delta^2+\nu_p\delta^2-2n\alpha^2(3\delta-4\alpha)^2\\
&=4n\alpha\delta^3-24n\alpha^2\delta^2+48n\alpha^3\delta-32n\alpha^4+\delta^2\nu_p\\
&=4n\alpha(\delta-2\alpha)^3+\delta^2\nu_p
\end{align*}
and, thus, the operator $\alpha\beta\Rcperp+\frac 14 \G_1+(\pi^*\omega)_1$ is positive if the cubic polynomial $4n\alpha(\delta-2\alpha)^3+\delta^2\nu$ is positive as well.
\end{proof}
In the unaltered case, or equivalently $\nu=0$, we can quantify the condition more nicely in terms of $\alpha\beta$.
\begin{cor}\label{CorN}
The operator $\alpha\beta\Rcperp+\frac14 \G_1$ is positive (semi-)definite if and only if $\alpha\beta>0$ ($\alpha\beta\geq 0$).
\end{cor}
\begin{proof}
In this case the determinant reads $4n\alpha(\delta-2\alpha)^3=\frac n2\alpha\beta^3$ which is positive (non-negative) if and only if $\alpha\beta >0$, ($\alpha\beta\geq 0$). In either case the trace satisfies
\begin{align*}
\delta^2+4n\alpha\delta-6n\alpha^2&\geq 4\alpha^2+8n\alpha^2-6n\alpha^2=2\alpha^2(n+2)>0.\qedhere
\end{align*}
\end{proof}

\begin{proof}[Proof of \autoref{stronglypos}]
We have seen that under $\Rcperp$, $\Rc_\mathrm{par}$, $\G_T$ and $\pi^*\omega$ the spaces $\Lambda^2_1$, $\Lambda^2_2$ and $\Lambda^2_3$ are invariant. Thus, we may decompose 
\begin{align*}
\Rc+\frac 14\G_T+\pi^*\omega=\left(\alpha\beta\Rcperp+\frac 14 \G_1+(\pi^*\omega)_1\right)\;\oplus\; \left(\Rcpar+(\pi^*\omega)_2\right)\;\oplus\; \frac 14\G_3.
\end{align*}
By assumption
\begin{align*}
\Rcpar+(\pi^*\omega)_2=\left(\Rc^{g_N}+\omega\right)|_{\mathcal{Q}^\perp}
\end{align*}
is positive (non-negative), where we have used the identification $\Lambda^2_2=\pi^*\mathcal{Q}^\perp$. The results so far are summarized in \autoref{intermediate}. In fact, in the non-negative case we are done.

\begin{table}[h]
\begin{align*}\arraycolsep=10pt\def\arraystretch{1.5}
\begin{array}{c|c|c|c|c|c}
\text{space} &\mathrm{dim} & \alpha\beta\Rcperp & \frac 14 \G_T & (\pi^*\omega)_1 &\Rcpar+(\pi^*\omega)_{2}\\
\hline
\Lambda^2_1 & 6 & \multicolumn{3}{c|}{\text{sum}>0\,\ (\text{resp.}\, \geq 0)} & 0  \\
\hline
\Lambda^2_2 & 6n-3+\binom{n}{2}\cdot 16& 0 & 0 & 0 & > 0\,\ (\text{resp.}\, \geq 0)\\
\hline
\Lambda^2_3 & n\cdot 12 & 0 & \geq 0 & 0 &0
\end{array}
\end{align*}
\caption{Positivity on summands of $\Lambda^2TM$}\label{intermediate}
\end{table}

In order to prove strong positivity we need to prove that $-\varepsilon\sigma_T$ provides strict positivity on the kernel of $\G_3$. For sufficiently small $\varepsilon$ it will do so without destroying positivity where already established. Indeed, the following \autoref{Pos4} shows that $\sigma_T$ is negative definite on the kernel of $\G_3$. 
\end{proof}
%
%
%\begin{rem}
%----------------------------------------------- I REMOVED THIS COMMENT. IT DOES NOT REALLY FIT WITH THIS NEW VERSION OF THE THEOREM
%\begin{enumerate}
%\item The converse statement that $\alpha\beta\geq 0$ if a given $3$-$(\alpha,\delta)$-Sasakian manifold is strongly non-negative is not shown yet, but follows directly from \autoref{sumonL1}.
%\item A similar strong non-positivity for $3$-$(\alpha,\delta)$-Sasakian manifolds with $\Rc^{g_N}+\omega\leq0$ does not hold. This will be proved as a byproduct of \autoref{S4}.
%\end{enumerate}
%\end{rem}
%
%
\begin{lemma}\label{Pos4}
%-----------------------------------------------
The operator $\Sc_T$ corresponding to $\sigma_T$ is negative definite on the kernel of $\G_3$ if and only if $\alpha\beta>0$.
\end{lemma}
\begin{proof}
%-----------------------------------------------
As in the proof of \autoref{S4} we may split $\G_3$ into $n$ copies of $\hat{\G}_3$ on each quaternionic subspace. Let $e_{4l},\ldots,e_{4l+3}\in\H_l$ be an adapted basis of one such subspace. Then from the same proof we find that $T(e_r\wedge\xi_i)=2\alpha\varphi_ie_l=-T(\varphi_je_r\wedge\xi_k)$. Thus,
\[
\ker \hat{\G}_3=\ker T\cap(\H_l\wedge\V)=\mathrm{span}\{e_r\wedge\xi_i+\varphi_je_r\wedge\xi_k\;\vert\; i=1,2,3;\; r=4l,\dots,4l+3\}.
\]
By definition of $\Sc_T$ we have
\begin{align*}
g(\Sc_T(e_r\wedge\xi_i),e_s\wedge\xi_a)&=g(T(e_r,\xi_i),T(e_s,\xi_a))+g(T(e_s,e_r),T(\xi_i,\xi_a))\\
&\qquad-g(T(e_s,\xi_i),T(e_r,\xi_a)).
\end{align*}
With $T=2\alpha\sum_{i=1}^n\eta_i\wedge\Phi^\H_i+2(\delta-4\alpha)\eta_{123}$ we compute each term individually
\begin{align*}
g(T(e_r,\xi_i),T(e_s,\xi_a))&=\begin{cases}4\alpha^2, & e_s=-\varphi_a\varphi_ie_r\\0\end{cases},\\
g(T(e_s,e_r),T(\xi_i,\xi_a))&=\begin{cases}2\alpha(\beta-4\alpha), & e_s=\varphi_i\varphi_ae_r\text{ and } a\neq i\\0\end{cases},\\
g(T(e_s,\xi_i),T(e_r,\xi_a))&=\begin{cases}4\alpha^2, & e_s=-\varphi_i\varphi_a e_r\\0\end{cases}.
\end{align*}
We thus obtain the full expression
\begin{align*}
\Sc_T(e_r\wedge\xi_i)&=4\alpha^2\sum_{a=1}^3(-\varphi_a\varphi_ie_r+\varphi_i\varphi_ae_r)\wedge\xi_a+2\alpha(\beta-4\alpha)\sum_{a\neq i}\varphi_i\varphi_ae_r\wedge\xi_a\\
&=2\alpha\beta\sum_{a\neq i}\varphi_i\varphi_ae_r\wedge\xi_a\\
&=2\alpha\beta(\varphi_ke_r\wedge\xi_j-\varphi_je_r\wedge\xi_k).
\end{align*}
Finally we compute $\Sc_T$ on $\ker \hat{\G}_3$ to obtain the result
\begin{align*}
\Sc_T(e_r\wedge\xi_i+\varphi_je_r\wedge\xi_k)&=2\alpha\beta(\varphi_ke_r\wedge \xi_j-\varphi_je_r\wedge\xi_k+\varphi_j\varphi_je_r\wedge\xi_i-\varphi_i\varphi_je_r\wedge\xi_j)\\
&=-2\alpha\beta(e_r\wedge\xi_i+\varphi_je_r\wedge\xi_k).\qedhere
\end{align*}
\end{proof}
As a word of caution we should state where this theorem might and might not be applicable. By assumption the quaternionic Kähler orbifold is strongly positive and thereby has positive sectional curvature. M.\ Berger investigated such manifolds in \cite{Berger66}. As observed in \cite{Dearricott04}, Berger's argument is purely local. It therefore extends to quaternionic Kähler orbifolds.
\begin{theorem}[\cite{Berger66},\cite{Dearricott04}]
%===================================================
Let $n\geq 2$ and $(M^{4n},g,\mathcal{Q})$ be quaternionic Kähler orbifold of positive sectional curvature. Then $(M^{4n},g,\mathcal{Q})$ is locally isometric to $\mathbb{H}P^n$ with its standard quaternionic Kähler structure.
\end{theorem}
Thus, the strong positivity result of \autoref{stronglypos} can only be applicable on $3$-$(\alpha,\delta)$-Sasaki manifolds of dimension $7$ or on finite quotients of $S^{4n+3}$. We will see in the next section that indeed both cases appear for homogeneous manifolds.
\subsection{The Homogeneous Case}
%=======================================
We would like to apply the positivity discussion to homogeneous $3$-$(\alpha,\delta)$-Sasaki manifolds, more precisely to those that fiber over Wolf spaces and their non-compact duals. We recall their construction from our previous publication \cite{ADS20}, extending the similar discussion for homogeneous $3$-Sasaki manifolds by \cite{Draperetall}.

\begin{df}
A triple $(G,G_0,H)$ is called generalized $3$-Sasaki data if $H\subset G_0\subset G$ are connected, real, simple Lie groups with Lie algebras $\mathfrak{h}\subset\mathfrak{g}_0\subset\mathfrak{g}$ such that:
\begin{enumerate}[i)]
\item $\mathfrak{g}_0=\mathfrak{h}\oplus\mathfrak{sp}(1)$ with $\mathfrak{sp}(1)$ and $\mathfrak{h}$ commuting subalgebras,
\item $(\mathfrak{g},\mathfrak{g}_0)$ form a symmetric pair, $\mathfrak{g}=\mathfrak{g}_0\oplus\mathfrak{g}_1$,
\item the complexification $\mathfrak{g}_1^\mathbb{C}=\mathbb{C}^2\otimes_\mathbb{C}W$ for some $\mathfrak h^{\mathbb{C}}$-module of $\dim_\mathbb{C}W=2n$,
\item $\mathfrak{h}^\mathbb{C}, \mathfrak{sp}(1)^\mathbb{C}\subset \mathfrak{g}_0^\mathbb{C}$ act on $\mathfrak{g}_1^\mathbb{C}$ by their respective action on $W$ and $\mathbb{C}^2$.
\end{enumerate}
\end{df}

\begin{theorem}[{\cite[Theorem 3.1.1]{ADS20}}]\label{construction}
%-----------------------------------
Consider some generalized $3$-Sasaki data $(G,G_0,H)$ and $0\neq\alpha,\delta\in\R$. Additionally suppose $\alpha\delta>0$ if
$G$ is compact and $\alpha\delta<0$ if $G$ is non-compact.

Let $\kappa(X,Y)=\mathrm{tr}(\mathrm{ad}(X)\circ\mathrm{ad}(Y))$ denote the Killing form on $\mathfrak g$.
Then define the inner product $g$ on the tangent space $T_pM=T_p(G/H)\cong\mathfrak{m}$ by
\begin{align*}
g\vert_{\V}&=\frac{-\kappa}{4\delta^2(n+2)},\qquad
g\vert_{\H}=\frac{-\kappa}{8\alpha\delta (n+2)},\qquad
\V\perp\H.
\end{align*}
Let $\xi_i=\delta \sigma_i\in\V=\mathfrak{sp}(1)$, where the $\sigma_i$ are the elements of
$\mathfrak{sp}(1)=\mathfrak{su}(2)$ given by
\begin{align*}
\sigma_1=
\begin{pmatrix}
i &0\\ 0&-i
\end{pmatrix},\quad
\sigma_2=
\begin{pmatrix}
0 &-1\\ 1&0
\end{pmatrix},\quad
\sigma_3=
\begin{pmatrix}
0&-i\\ -i &0
\end{pmatrix}.
\end{align*}
Define endomorphisms $\varphi_i\in\mathrm{End}_{\mathfrak{h}}(\mathfrak{m})$ for $i=1,2,3$ by
\begin{align*}
\varphi_i\vert_{\V}&=\frac 1 {2\delta} \mathrm{ad}\,\xi_i,\qquad
\varphi_i\vert_{\H}=\frac 1 \delta \mathrm{ad}\,\xi_i.
\end{align*}
Together with $\eta_i=g(\xi_i,\cdot)$ the collection $(G/H,\varphi_i,\xi_i,\eta_i,g)$ defines a
homogeneous $3$-$(\alpha,\delta)$-Sasaki structure.
\end{theorem}

Note that the homogeneous $3$-$(\alpha,\delta)$-Sasaki structure on $\mathbb{R}P^{4n+3}$ is not directly obtained by this construction but as the quotient of $S^{4n+3}=\mathrm{Sp}(n+1)/\mathrm{Sp}(n)$ by $\mathbb{Z}_2$. Here the local structure is the same as for $S^{4n+3}$ given in the theorem. With this exception we have that all positive homogeneous $3$-$(\alpha,\delta)$-Sasaki manifolds are obtained from the theorem. In the negative case more exist so we will restrict ourselves in the following discussion to those over symmetric base spaces.

\begin{theorem}
%-----------------------------------------------
Let $M=G/H$ be a homogeneous $3$-$(\alpha,\delta)$-Sasaki manifold.
\begin{enumerate}[a)]
\item If $M$ is a positive $3$-$(\alpha,\delta)$-Sasaki manifold then the canonical curvature operator $\Rc$ is non-negative if and only if $\alpha\beta\geq 0$. In this case $M$ is strongly non-negative.
\item If $M$ is a negative $3$-$(\alpha,\delta)$-Sasaki manifold over a symmetric base then the canonical curvature operator $\Rc$ is non-positive.
\end{enumerate}
\end{theorem}
\begin{proof}
%-----------------------------------------------
In the positive case $M$ has to fiber over a symmetric base, compare \cite{ADS20}. In this case the base is a compact symmetric space, hence the curvature operator $\Rc^{g_N}$ is non-negative. In part b) the base is a non-compact symmetric space by assumption, hence $\Rc^{g_N}\leq 0$. Therefore in both cases it fulfills the requirement of \autoref{semidefinite}. In the positive case also \autoref{nonnegop} applies.
\end{proof}
%\[
%\langle\Rc_B(X\wedge Y),Z\wedge W\rangle=\frac{-1}{8\alpha\delta(n+2)}\kappa([[X,Y],Z],W)=\frac{-1}{8\alpha\delta(n+2)}\kappa([X,Y],[Z,W])
%\]
%for $X,Y,Z,W\in\H\cong T_oN$. Since $[X,Y],[Z,W]\in\mathfrak{h}\oplus\V=\mathfrak{g}_0$ the Killing form acts negatively on the commutators in both the positive and negative $3$-$(\alpha,\delta)$-Sasakian case.
%\begin{enumerate}[a)]
%\item Thus $\Rc_B\geq 0$ in the positive case as $\alpha\delta>0$. It follows that $\Rc_\mathrm{par}$ is also non-negative. Now, if $\alpha\beta\geq0$ also $\alpha\beta\Rcperp\geq0$ proving the first claim.
%\item Analogously we have $\Rc_B\leq 0$ if $\alpha\delta<0$ and, thus, $\Rc_\mathrm{par}$ is non-positive. Also for $\alpha\delta<0$ we have $\alpha\beta<0$ and thus always $\alpha\beta\Rcperp\leq0$. Therefore, $\Rc$ is non-positive in the negative $3$-$(\alpha,\delta)$-Sasakian case.\qedhere
%\end{enumerate}
%
We will next focus on strong positivity. This is much more restrictive than strong non-negativity. In particular, strong positivity implies strict positive sectional curvature and homogeneous manifolds with strictly positive sectional curvature have been classified \cite{Wallach}\cite{WilZil}\cite{BerBer}. Out of these only the $7$-dimensional Aloff-Wallach-space $W^{1,1}$, the spheres $S^{4n+3}$ and real projective spaces $\mathbb{R}P^{4n+3}$ admit homogeneous $3$-$(\alpha,\delta)$-Sasaki structures. We will thus prove
\begin{theorem}\label{StrictPos}
%-----------------------------------------------
The $3$-$(\alpha,\delta)$-Sasaki manifolds
\begin{enumerate}[a)]
\item $W^{1,1}=\mathrm{SU}(3)/S^1$ with $4$-form $-(\frac 14+\varepsilon)\sigma_T$ for small $\varepsilon>0$,
\item $S^{4n+3}$, $\mathbb{R}P^{4n+3}$, $n\geq 1$, with $4$-form $\frac{\alpha\delta}{4}\pi^*\Omega_N-(\frac 14+\varepsilon)\sigma_T$ for small $\varepsilon>0$
\end{enumerate}
where $\pi^*\Omega_N\coloneqq \sum_{i=1}^3\Phi_i^\H\wedge\Phi_i^H$, i.e.\ $\Omega_N$ is the fundamental $4$-form of the qK base, are strongly positive if and only if $\alpha\beta>0$.
\end{theorem}
\begin{rem}
%-----------------------------------------------
The strong positivity of these spaces $W^{1,1}$ and $S^{4n+3}$, $\mathbb{R}P^{4n+3}$, can actually be proven by the Strong Wallach Theorem in \cite{StrPosCurv}. We compare to our case:
\begin{enumerate}[i)]
\item Observe that all positive homogeneous $3$-$(\alpha,\delta)$-Sasaki manifolds are given by a homogeneous fibration
\[
\mathrm{SO}(3)=G_0/H\rightarrow G/H \rightarrow G/G_0.
\]
In the case of $S^{4n+3}$ the fiber is $\mathrm{Sp}(1)$ instead.
\item In their strong Wallach theorem \cite{StrPosCurv} the autors consider the metrics $g_t=tQ|_{\V}+Q|_\H$ for $0<t<1$, where $Q$ is a negative multiple of the Killing form. If we set $Q=\frac{-\kappa}{8\alpha\delta(n+2)}$ as in the $3$-$(\alpha,\delta)$-Sasaki setting then $t=\frac{2\alpha}{\delta}$ and, thus, the condition $0<t<1$ is equivalent to $\beta>0$.
\item We have $\dim G_0/H=3$ and $G_0/H=\mathrm{SO}(3)=\mathbb{R}P^3$, $S^3$ in the case of $S^{4n+3}$, with a scaled standard metric. In particular the fiber is of positive sectional curvature.
\item They require a strong fatness property for the homogeneous fibration. Adapted to our notation the bundle is strongly fat if there is a $4$-form $\tau$ such that $F+\tau\colon \H\otimes\V\rightarrow\H\otimes\V$ is positive definite, where $F$ is given by
\begin{align*}
g(F(X\wedge \xi_i),Y\wedge \xi_j)&=g([X,\xi_i],[Z,\xi_j])=\delta^2g(\varphi_iX,\varphi_jY)\\
&=\frac{\delta^2}{4\alpha^2}g(T(X\wedge \xi_i),T(Y\wedge \xi_j))=\frac{\delta^2}{4\alpha^2}g(\G_T(X\wedge \xi_i),Y\wedge \xi_j).
\end{align*}
Thus by the previous lemma $\tau=-\varepsilon\sigma_T$ accomplishes strong fatness for sufficiently small $\varepsilon$.
\item The final condition is for the base to be one of $S^{4n},\mathbb{R}P^{4n},\mathbb{C}P^{2n},\mathbb{H}P^n$. The only homogeneous $3$-$(\alpha,\delta)$-Sasaki manifolds such that this holds are $S^{4n+3}$, $\mathbb{R}P^n$ which fiber over $\mathbb{H}P^n$, and $W^{1,1}$ which fibers over $\mathbb{C}P^2$.
\end{enumerate}
Note that i)-iv) are valid for all positive homogeneous examples not only for the spheres, real projective spaces and $W^{1,1}$.
\end{rem}
\begin{proof}[Proof of \autoref{StrictPos}]
%-----------------------------------------------
Since our discussion is pointwise we will identify tensors on $N$ with those on $\H$.
\begin{enumerate}[a)]
\item Let $(M,\varphi_i,\xi_i,\eta_i,g)_{i=1,2,3}$ be a $7$-dimensional $3$-$(\alpha,\delta)$-Sasaki manifold and $\pi\colon M\to N$ its canonical submersion. Then $N$ is $4$-dimensional and $\Phi_i^\H\wedge\Phi_i^\H=2\mathrm{dVol}_N$ for each $i=1,2,3$. By \eqref{dT}
\[
\sigma_T\vert_{\Lambda^4\H}=\frac12\mathrm{d}T\vert_{\Lambda^4\H}=2\alpha^2\sum_{i=1}^3\Phi_i^\H\wedge\Phi_i^\H=12\alpha^2\mathrm{dVol}_N.
\]
Now the $6$-dimensional space $\Lambda^2 N$ splits as usual into the spaces of self-dual and anti self-dual forms $\Lambda^2_+$ and $\Lambda^2_-$. In other words these are the $\pm1$-eigenspaces of $\mathrm{dVol}_N$ as an operator $\Lambda^2N\to\Lambda^2N$. Checking in an adapted basis we find that $\mathrm{span}\{\Phi_i^\H\}=\Lambda^2_1\cap\Lambda^2\H=\Lambda^2_+$ and $\Lambda^2_2=\Lambda^2_-$. Now suppose that $N$ is a compact symmetric space. Then $\Rc^{g_N}\geq 0$ and if $\alpha\delta>0$ its restriction $\Rc^{g_N}|_{\Lambda^2_+}>0$ is strictly positive. Indeed,  \eqref{RgNdecomp} shows this, since $\Lambda^2_1\subset\ker\Rcpar$ (compare \autoref{propRpar}). Now for sufficiently small $\varepsilon>0$ the operator $\Rc^{g_N}-\varepsilon\mathrm{dVol}_N>0$. Hence we may apply \autoref{stronglypos}. Note that $\varepsilon$ and therefore the lower bound $\nu$ of eigenvalues of $-\varepsilon\mathrm{dVol}_N$ can be chosen arbitrarily small. Since $\alpha\beta>0$ by the proof of \autoref{CorN} we have $\delta^2+4n\alpha\delta-6n\alpha^2>0$ and $4n\alpha(\delta-2\alpha)^3>0$, hence also 
\begin{align*}
\delta^2+4n\alpha\delta-6n\alpha^2+\nu>0\quad\text{and}\quad4n\alpha(\delta-2\alpha)^3+\delta^2\nu>0.
\end{align*}
Then \autoref{stronglypos} finishes the proof.

\item Now the $4$-form $\omega=\frac{\alpha\delta}{4}\Omega_N$ appears as possible $4$-form in the strong positivity of $\mathbb{H}P^n$. In \cite{StrPosCurv} they prove that the $4$-form $\mathfrak{b}(\rho)$ suffices where $\rho$ is given as the symmetric product of twice the $A$-tensor of the submersion $(S^{4n+3},g_0)\rightarrow (\mathbb{H}P^n,g_B)$, $(\mathbb{R}P^{4n+3},g_0)\rightarrow(\mathbb{H}P^n,g_B)$ respectively\footnote{Due to conflicting notation, we renamed this $\rho$ from $\alpha$ in \cite{StrPosCurv}.}. Here the metric $g_0$ denotes up to global scaling by $\frac{1}{8\alpha\delta(n+2)}$ the standard round metric. Adapted to our notation
\[
\rho(X\wedge Y,Z\wedge V)=g_0(A_XY,A_ZV)=\frac 14 g_0([X,Y]_\mathfrak{m},[Z,V]_\mathfrak{m})=\frac{\alpha\delta}{2}\sum_{i=1}^3\Phi_i^\H(X,Y)\Phi_i^\H(Z,V).
\]
and thus
\[
\mathfrak{b}(\rho)=\frac{\alpha\delta}{2}\sum_{i=1}^3\mathfrak{b}(\Phi_i^\H\otimes\Phi_i^\H)=\frac{\alpha\delta}{4}\sum_{i=1}^3\Phi_i^\H\wedge\Phi_i^{\H}=\frac{\alpha\delta}{4}\pi^*\Omega_N.
\]
Therefore we have $\Rc^{g_N}+\frac{\alpha\delta}{4}\Omega_N>0$ on $\Lambda^2N$. It remains to check that $\mathcal{Q}$ is an eigenspace of $\omega$. We compute for a basis $e_1,\dots,e_{4n}$ of $TN$
\begin{align*}
\omega(\Phi_i^\H)&=\frac{\alpha\delta}{4}\sum_{s=1}^3\Phi_i^\H\intprod(\Phi_s^\H\wedge\Phi_s^\H)\\
&=\frac{\alpha\delta}{2}\left(2n\Phi_i^\H-\frac 12 \sum_{l=1}^{4n} \sum_{s=1}^3 \left((e_l\intprod\Phi^\H_s)\wedge(\varphi_ie_l\intprod\Phi_s^\H)\right)\right)\\
&=\frac{\alpha\delta}{2}\left(2n\Phi_i^\H+\frac 12 \sum_{l=1}^{4n} \sum_{s=1}^3 \varphi_se_l\wedge \varphi_s\varphi_ie_l\right)\\
&=\frac{\alpha\delta}{2}\left(2n\Phi_i^\H+\frac 12 \sum_{l=1}^{4n} (e_l\wedge \varphi_ie_l-\varphi_je_l\wedge \varphi_ke_l+\varphi_ke_l\wedge\varphi_j e_l\right)\\
&=\frac{\alpha\delta}{2}\left(2n\Phi_i^\H+\frac 12 \sum_{l=1}^{4n} (e_l\wedge \varphi_ie_l-2\varphi_je_l\wedge \varphi_i\varphi_je_l\right)\\
&=\frac{\alpha\delta}{2}(2n+1)\Phi_i^\H.
\end{align*}
Since the eigenvalue $\nu>0$ is positive $\alpha\beta\Rcperp+\frac 14\G_1+(\pi^*\omega)_1\geq\alpha\beta\Rcperp+\frac 14\G_1$ and $\alpha\beta>0$ suffices as argued in \autoref{CorN}.
\end{enumerate}
\end{proof} 

\subsection{Some Inhomogeneous Example}\label{inhomogeneous}
Let us recall $3$-Sasaki reduction introduced in \cite{BGM94}. Let $(M,\varphi_i,\xi_i,\eta_i,g)$ be a $3$-Sasaki manifold and $G$ a connected compact Lie group acting on $M$ by $3$-Sasaki automorphisms. We consider the $3$-Sasaki moment map
\begin{align*}
\mu\colon M&\to\mathfrak{g}^*\otimes\R^3\\
x&\mapsto (X\to\eta_i(\overline{X}_x))_{i=1,2,3}
\end{align*}
where $\overline{X}_x$ is the fundamental vector field of $X\in \mathfrak{g}$ at $x\in M$.
\begin{theorem}[\cite{BGM94}]\label{reduction}
Assume that $0$ is a regular value of $\mu$ and that $G$ acts freely on the preimage $\mu^{-1}(\{0\})$. Denote the embedding $\iota\colon \mu^{-1}(\{0\})\to M$ and the submersion $\pi\colon \mu^{-1}(\{0\})\to \mu^{-1}(\{0\})/G$. Then $(\mu^{-1}(\{0\})/G, \check{\varphi}_i,\check{\xi}_i,\check{\eta}_i,\check{g})_{i=1,2,3}$ is a smooth $3$-Sasaki manifold, where the $3$-Sasaki structure is uniquely determined by $\iota^*g=\pi^*\check{g}$ and $\check{\xi}_i=\pi_*(\xi_i|_{\mu^{-1}(\{0\})})$.
\end{theorem}

We will focus on actions of $S^1$ on the $3$-Sasaki sphere $S^{11}\subset \mathbb{H}^3$ via 
\begin{align*}
z\cdot(q_1,q_2,q_3)=(z^{p_1}q_1,z^{p_2}q_2,z^{p_3}q_3).
\end{align*}
In \cite[Theorem $13.7.6$]{Boyer&Galicki} the authors show that for pairwise coprime, positive integers $p_1,p_2,p_3$ these actions satisfy the assumptions in \autoref{reduction} and, thus, give rise to $3$-Sasaki manifolds $\mathcal{S}(p_1,p_2,p_3)$. If $p_1=p_2=p_3=1$ this is exactly the homogeneous $3$-Sasaki Aloff-Wallach space $W^{1,1}$. Apart form this they are shown in \cite[Corollary $13.7.13$]{Boyer&Galicki} to be of cohomogeneity $1$ or $2$.

In \cite{Dearricott04} the author shows that under the assumption $\sqrt{2}\min p_i> \max p_i$ a certain deformation of the $3$-Sasaki metric, corresponding to a $\H$-homothetic deformation in our notation, admits positive sectional curvature. We make use of a key step of his showing that their underlying quaternionic Kähler orbifolds have positive sectional curvature, \cite[Theorem 2]{Dearricott04}.
\begin{theorem}[\cite{Dearricott04}]
Let $\mathcal{S}(p_1,p_2,p_3)$ be as before and $\mathcal{O}(p_1,p_2,p_3)$ the underlying quaternionic Kähler orbifold. If $\sqrt{2}\min p_i>\max{p_i}$ then $\mathcal{O}(p_1,p_2,p_3)$ has positive sectional curvature.
\end{theorem}
In order to make the jump from positive sectional curvature to strongly positive curvature we make use of the fact that $\mathcal{O}(p_1,p_2,p_3)$ is $4$-dimensional. In this dimension Thorpe proves the following, \cite[Corollary $4.2$]{Thorpe71}.
\begin{theorem}[\cite{Thorpe71}]
Let $V$ be a $4$-dimensional vector space and $R$ any algebraic curvature operator on $V$. If $\lambda$ is the minimal sectional curvature of $R$ then there is a unique $\omega\in \Lambda^4 V$ such that $\lambda$ is the minimal eigenvalue of $R+\omega$.
\end{theorem}

We are finally ready to state our main theorem.
\begin{theorem}
Let $p_1,p_2,p_3$ be coprime integers with $\sqrt{2}\min p_i>\max p_i$. Then there is a $\H$-homothetic deformation of $\mathcal{S}(p_1,p_2,p_3)$ that has strongly positive curvature.
\end{theorem}
\begin{proof}
Thorpe's theorem proves that the orbifold $\mathcal{O}(p_1,p_2,p_3)$ has not only positive sectional curvature but strongly positive curvature. Since we are in dimension $4$ the form $\omega$ is necessarily a multiple of the volume form $\omega=\nu_p\mathrm{dVol}$. As before the volume form has eigenspaces $\Lambda^2_\pm$ where $\Lambda^2_+=\mathcal{Q}$ and $\Lambda^2_{-}=\mathcal{Q}^\perp$. In particular, $\omega$ is an adapted $4$-form with minimal eigenvalue $\nu=\min\nu_p$. The minimum exists since the orbifolds are quotients of compact spaces and, thus, compact themselves. All in all we may apply \autoref{stronglypos}. Note that by \autoref{nonscalingHdeform} we obtain a $\H$-homothetic deformation of $\mathcal{S}(p_1,p_2,p_3)$ with $\alpha/\delta\gg 0$ sufficiently big while not changing the metric $g_\mathcal{O}$ on $\mathcal{O}(p_1,p_2,p_3)$.
\end{proof}

%\input{SubmersionCurvature}
%\newpage
%\input{To-Dos}
%\newpage
%
%\phantom{x}
%
%\newpage

%-----------------------------------------------------------------------------------------
%
\noindent
Ilka Agricola, Fachbereich Mathematik und Informatik,
Philipps-Universit\"at Marburg,
Campus Lahnberge, 35032 Marburg, Germany.
\texttt{agricola@mathematik.uni-marburg.de}

\bigskip\noindent
Giulia Dileo, Dipartimento di Matematica, Universit\`a degli Studi di Bari Aldo Moro,
Via E. Orabona 4, 70125 Bari, Italy.
\texttt{giulia.dileo@uniba.it}

\bigskip\noindent
Leander Stecker, Fachbereich Mathematik,
Universität Hamburg,
Bundesstraße 55, 20146 Hamburg, Germany.
\texttt{leander.stecker@uni-hamburg.de}

\end{document}